\documentclass[11pt]{article}

\title{The corona algebra of the stabilized Jiang-Su algebra}

\author{Huaxin Lin and Ping Wong Ng}

\date{}

\usepackage{amssymb, amsthm, amsmath, amscd}
\usepackage[usenames,dvipsnames]{color}

\newtheorem{thm}{Theorem}[section]
\newtheorem{lem}[thm]{Lemma}
\newtheorem{prop}[thm]{Proposition}
\newtheorem{cor}[thm]{Corollary}
\newtheorem{NN}[thm]{}
\theoremstyle{definition}\newtheorem{df}[thm]{Definition}
\theoremstyle{definition}\newtheorem{rem}[thm]{Remark}
\theoremstyle{definition}

\renewcommand{\phi}{\varphi}

\newcommand{\N}{\mathbb{N}}
\newcommand{\Z}{\mathbb{Z}}
\newcommand{\Q}{\mathbb{Q}}

\newcommand{\C}{\mathbb{C}}
\newcommand{\T}{\mathbb{T}}
\newcommand{\K}{\mathcal{K}}

\newcommand{\Aff}{\operatorname{Aff}}

\newcommand{\morp}{contractive completely positive linear map}

\newcommand{\hm}{homomorphism}
\newcommand{\dt}{\delta}
\newcommand{\ep}{\epsilon}

\newcommand{\andeqn}{\,\,\,{\rm and}\,\,\,}
\newcommand{\rforal}{\,\,\,{\rm for\,\,\,all}\,\,\,}
\newcommand{\CA}{$C^*$-algebra}
\newcommand{\SCA}{$C^*$-subalgebra}

\newcommand{\af}{{\alpha}}

\newcommand{\beq}{\begin{eqnarray}}
\newcommand{\eneq}{\end{eqnarray}}
\newcommand{\tforal}{\,\,\,\text{for\,\,\,all}\,\,\,}
\newcommand{\tand}{\,\,\,\text{and}\,\,\,}
\newcommand{\aff}{{\rm Aff}}
\newcommand{\p}{{\mathfrak{p}}}
\newcommand{\q}{{\mathfrak{q}}}
\newcommand{\fr}{{\mathfrak{r}}}

\begin{document}

\maketitle

\begin{abstract}
Let ${\cal Z}$ be the Jiang-Su algebra and ${\cal K}$ the \CA\, of compact operators on an infinite dimensional separable Hilbert space.
We prove that the corona algebra $M({\cal Z}\otimes {\cal K})/{\cal Z}\otimes {\cal K}$ has real rank zero. We actually prove a more general
result.
\end{abstract}

\section{Introduction}

The Jiang-Su algebra ${\cal Z}$ is a projectionless unital simple separable infinite dimensional amenable \CA\, with the same $K$-theory as that of the complex field $\C$ (see \cite{JS}).
It had been an interesting problem to find projectionless unital infinite dimensional simple \CA s
(\cite{Kaplansky}, \cite{BlackadarNonunital}, \cite{BlackadarUnital},
\cite{ConnesDyn}). However, the Jiang-Su algebra is the nicest that one can
get. In fact ${\cal Z}$ is an inductive limit of sub-homogeneous \CA s.
It is one of many products of  the decade of 1990's in the Elliott program, the program of classification of amenable \CA s.  However, more recently,
${\cal Z}$ has played a much important role in the study of  \CA s. (See,
for example, \cite{WinterLocalize}, \cite{WinterZ}, 
\cite{Lninv}.). It becomes
increasingly
important to fully understand the structure of the Jiang-Su algebra.

In many ways, ${\cal Z}$ behaves like $\C,$ as its $K$-theory suggests.
One indirect, but important feature, of the unital simple \CA\ $\C$, is
the simplicity (as well as complexity)  of the Calkin algebra $M(\C\otimes {\cal K})/\C\otimes {\cal K}.$  Let $A$ be a non-unital but $\sigma$-unital \CA.
The quotient $M(A)/A$ is called the corona algebra of $A.$ The Calkin algebra is the corona algebra of $\C\otimes {\cal K},$
the stablization of $\C.$ One of the important consequences of the simplicity of the Calkin algebra is that it has real rank zero. As we know, many important results  in operator algebras (as well as operator theory) are related to the Calkin algebra, such as
Fredholm index theory, the BDF-theory, \CA\ extension theory and the $KK$-theory. The fact that the Calkin algebra has real rank zero plays a crucial role in all these even though the term real rank zero was invented much later
(\cite{BrownPedersenRR0}).
From that point of view, the corona algebra $M({\cal Z}\otimes {\cal K})/{\cal Z}\otimes {\cal K},$ inevitably, is also important, given the central
role ${\cal Z}$  is playing in the study of classification of amenable \CA s.
By now, we know that the corona algebra of ${\cal Z}\otimes {\cal K}$
is not simple. In fact it has a unique proper closed ideal. Nevertheless, we
would like to know whether the corona algebra of ${\cal Z}\otimes {\cal K}$
has some other similar structure to that of $\C\otimes {\cal K}.$  The problem
whether the corona algebra $M({\cal Z}\otimes {\cal K})/{\cal Z}\otimes {\cal K}$ has real rank zero has been around for many years and was specifically
raised at the
American Institute of Mathematics in 2009, as well as other places
(\cite{KucerovskyNgPerera}).  Moreover, there is a long history of similar
questions
(e.g., (\cite{BrownPedersenRR0}, \cite{ZhCoronaI}, \cite{HigsonRordam},
\cite{LinWVNI},
\cite{LinWVNII},
\cite{LinWVNIII},
\cite{LinOsaka},
\cite{ZhangWVN}, \cite{ZhangWVNII}, \cite{Murphy}).
The purpose of this paper is to establish that the corona algebra of ${\cal Z}\otimes {\cal K}$ does have real rank zero, despite the fact it is not simple and the fact that ${\cal Z}$  has no proper projection.

One of the earliest applications of $K$-theory in operator algebras
is the proof of the following theorem: an extension of AF-algebras is again an AF-algebra (\cite{BrownLift} and  \cite{ElliottLift}).
This extends to a more general form: Let $0\to B\to E\to A\to 0$ be a short exact sequence of \CA s, where the ideal and quotient have real rank zero. Then
$E$ has real rank zero if and only if $K_1(B)$ is trivial (see \cite{BrownPedersenRR0}; also see \cite{ZhCoronaI}). Denote by $Q({\cal Z})$ the corona algebra of ${\cal Z}\otimes {\cal K}.$ Let $J$ be the closed ideal
of $M({\cal Z}\otimes {\cal K})$ such that $\pi(J)$ is the unique proper
nontrivial
ideal of the corona algebra, where $\pi: M({\cal Z}\otimes {\cal K})\to
Q({\cal Z})$ is the quotient map. We have that
$Q({\cal Z})/\pi(J)$ is purely infinite and simple (this is well-known;
an explicit reference is \cite{KucerovskyNgPerera}; it
also follows
immediately from \cite{LinFull} Theorem 3.5;
see also \cite{ZhStruct} Theorem 2.2 and its proof).
By a result of S. Zhang (\cite{ZhPI}), it has real rank zero. It is also known that
$\pi(J)$ is also purely infinite and simple (this is also well-known;
an explicit reference is \cite{KucerovskyNgPerera};
it
also follows
immediately from the definition of $J$ in \cite{LinContScale} 2.2, Remark 2.9
and
Lemma 2.1; see also  \cite{ZhStruct} Theorem 2.2 and its proof).
Therefore, from the above mentioned result, if  $K_1(J)$ is trivial, then $Q({\cal Z})$ has real rank zero. Much of the work of this study is to show just that.

The general strategy to show that $K_1(J)$ is trivial is taken from
an original idea of Elliott (\cite{Elliott1974}). However, unlike the case that Elliott considered, $A\otimes {\cal K}$ (in particular, in the case that $A={\cal Z}$) lacks sufficiently many projections. Instead, we need to use positive elements
and Cuntz comparison. The method to adapt Cuntz's relation to compare positive elements in the multiplier algebra of a non-unital simple \CA\, was initially used in \cite{LinContScale} (and \cite{Lncs2}). When $A$ has real rank zero,
one has the uniform bound for the length of path of unitaries in $A$ which connects to the identity.
 This fact  plays an  
important role in many of the earlier studies of multiplier algebras. In general, however,
for a unital simple \CA\, $A,$ the exponential length could be infinite.
In other words, there will not be any control of the length.
The idea to  controlling the exponential length of unitaries  in the multiplier algebras via the closure of the commutator subgroup instead of controlling the exponential length in $U_0(A)$ directly is new. We need results of exponential length of unitaries in the unitization of a non-unital
simple \CA.
The method to actually controlling the exponential length of unitaries in the closure of commutators for unital simple \CA s was taken from \cite{Lnexp}
which is based  on the results in the connection to the Elliott program of classification of amenable simple \CA s.

The paper is organized as follows.
Section 2 serves the purpose of setting up notation and terminology that
will be used in subsequent sections.  
Section 3 contributes to the understanding of exponential length of unitaries in the closure of commutator subgroup
of non-unital simple \CA s. The computation of the exponential length plays critical role in dealing with unitary group
of some ideals in the multiplier algebras and corona algebras of some simple \CA s, in particular, those of ${\cal Z}\otimes {\cal K}.$
Section 4 contains a number of technical lemmas that use the Cuntz relation to produce and deal with projections
in the multiplier algebra of certain non-unital simple \CA s.  The main technical lemma is \ref{cornerunit} which allows
us to connect a unitary to one which is nontrivial only in
a small corner.
Section 5 contains the main result.  We show that the corona algebra of $A\otimes {\cal K}$ has real rank zero for a class of unital simple
\CA s $A$ which includes the projectionless simple \CA\, ${\cal Z}.$
We expect a number of direct applications of the main result that $M({\cal Z}\otimes {\cal K})/{\cal Z}\otimes {\cal K}$  has real rank zero. However, these would be
the subject of future projects.

{\bf Acknowldgments}  This project began in 2009 when both authors visited American Institute of Mathematics during the Workshop
on the Cuntz semigroup. It restarted in May 2012 when both authors were in 
the Research Center for Operator Algebras in
East China Normal University and partially supported by the Center.
We thank Leonel Robert for pointing out that the proof of Proposition
\ref{prop:RangeOfTrace} can be shortened. 

\section{Notations}

\begin{df}\label{Dtrace}
{\rm
For any \CA\, $C,$ denote by $T(C)$ the tracial state space of $C.$
Let $\tau\in T(C)$ and let $n\ge 1$ be an integer. We will also use $\tau$
for the trace $\tau\otimes Tr,$ where $Tr$ is the standard (non-normalized)
trace on $M_n.$

}
\end{df}

\begin{df}\label{Dcu}
Let $A$ be a unital \CA. Denote by $U(A)$ the unitary group of $A,$ by
$U_0(A)$ the path connected component of $U(A)$ containing the identity.
Denote by $CU(A)$ the {\it closure} of the commutator group of $U_0(A).$
\end{df}

\begin{df}\label{Ddet}

{\rm
Let $A$ be a unital \CA\, with $T(A)\not=\emptyset.$
Let $u\in U_0(A).$
Suppose that $\{u(t): t\in [0,1]\}$ is a continuous path of unitaries
which is also piecewise smooth such that $u(0)=u$ and $u(1)=1.$
Define the de la Harpe-Skandalis determinant as follows:
\beq\label{Ddet-1}
{\rm Det}(u)(\tau):={\rm Det}(u(t))(\tau):=\int_{[0,1]}\tau({du(t)\over{dt}}u(t)^*) dt\tforal \tau\in T(A).
\eneq
Note that, if $u_1(t)$ is another continuous path which is piecewise smooth with $u_1(0)=u$ and $u_1(1)=1,$ then
${\rm Det}((u(t))-{\rm Det}(u_1(t))\in \rho_A(K_0(A)).$
Suppose that $u, v\in U(A)$ and $uv^*\in U_0(A).$
Let $\{w(t): t\in [0,1]\}\subset U(A)$ be a piecewise smooth
and continuous path such that $w(0)=u$ and $w(1)=v.$
Define
$$
R_{u,v}(\tau)={\rm Det}(w(t))(\tau)=\int_{[0,1]}\tau({dw(t)\over{dt}}w(t)^*) dt\tforal \tau\in T(A).
$$
Note that $R_{u,v}$ is well-defined (independent of the choices of the path) up to elements in $\rho_A(K_0(A)).$

}

\end{df}

\begin{df}\label{DQ}
Let $\p$ be a supernatural number. Denote by $M_\p$ the UHF-algebra associated with $\p.$ Denote by $\Q_\p$ the $K_0(M_\p)$ which is identified with a subgroup of $\Q.$ Denote by $Q$ the UHF-algebra
with $K_0(Q)=\Q$ and $[1_Q]=1.$

\end{df}

\begin{df}\label{DTR=0A0Z}
{\rm  Let $A$ be a unital simple \CA. We write
$TR(A)=0$ if the tracial rank of $A$ is zero (see \cite{Lntr}).

Denote by ${\cal A}_0$ the class of unital simple \CA s such that
$TR(A\otimes U)=0$ for some infinite dimensional UHF-algebra $U.$
It follows from \cite{LS} that if $A\in {\cal A}_0$ then
$TR(A\otimes U)=0$ for all infinite dimensional simple AF-algebras $U.$

 Denote by ${\cal Z}$ the Jiang-Su algebra.  It is 
a projectionless simple ASH-algebra with $K_0({\cal Z})=\Z$ and $K_1({\cal Z})=0.$
 Note that ${\cal Z}$ has a unique tracial state and ${\cal Z}\in {\cal A}_0.$

 A unital separable \CA\, $A$ is said to be ${\cal Z}$-stable if
 $A\otimes {\cal Z}\cong A.$

 Let $\p$ and $\q$ be two relative prime supernatural numbers.
 Define
 $$
 {\cal Z}_{\p, \q}=\{f\in C([0,1], M_{\p\q}):
 f(0)\in  M_\p \otimes 1_{M_\q}\andeqn f(1)\in 1_{M_\p}\otimes M_\q\}.
 $$
 }
 \end{df}

  By Theorem 3.4 of \cite{RorWinter},
there is a trace-collapsing unital *-embedding
$\phi : \mathcal{Z}_{\p, \q} \rightarrow \mathcal{Z}_{\p, \q}$.
Moreover, $\mathcal{Z}$ is the stationary C*-inductive limit
$\mathcal{Z} = \lim (\mathcal{Z}_{\p, \q}, \phi)$
(each building block is isomorphic
to $\mathcal{Z}_{\p, \q}$ and each connecting map is the same
as $\phi$).

\begin{NN}\label{cuntz}
{\rm
Let $A$ be a unital C*-algebra and let
 $\tau$ be  a tracial state of $A$.
Let $a \in M_n(A)_+$.
Then $d_{\tau}(a) =_{df} \lim_{n\rightarrow \infty} \tau(a^{1/n})$.

For a C*-algebra $A$ and $b, c \in A_+$,
$b \preceq c$ if $b$ is Cuntz subequivalent to $c$, i.e., there exists
a sequence $\{ x_n \}$ in $A$ such that $x_n c x_n^* \rightarrow b$.

A simple \CA\, $A$ with $T(A)\not=\emptyset$ is said to have strict comparison for positive element (by traces),
if, for any two positive elements $a, b\in M_n(A),$ $a\lesssim b$ whenever 
$d_\tau(a)<d_\tau(b)$ for all $\tau\in T(A).$ 
}
\end{NN}

\begin{NN}\label{MinIdeal}
{\rm Let $C$ be a non-unital but $\sigma$-unital  nonelementary simple
C*-algebra.  In \cite{LinContScale}, it is proven that there exists
a unique smallest ideal $J$ of $M(C)$ which properly contains $C$.
$J$ has the form $J = \overline{J_0}$ where
$$J_0 = \{ x \in M(C) : \forall a\in C_+ - \{ 0 \}, \exists n_0 \ni
(e_m - e_n) x^* x (e_m - e_n) \preceq a, \forall m > n \geq n_0 \}$$
In the above, $\{ e_n \}$ is an approximate identity for $C$.
$J$ is independent of the choice of the approximate identity $\{ e_n \}$.
(See \cite{LinContScale} 2.2, Lemma 2.1 and Remark 2.9.
See also \cite{Lncs2}.)

If, in addition,
$C$ has the  form $C= A \otimes \K$ where $A$ is unital C*-algebra with
unique tracial state $\tau$ and if also $C$ has strict comparison of
positive elements,
then $J$ (as defined above) is the unique ideal of
$M(C)$ which sits properly between
$C$ and $M(C)$.  Moreover, in this case, $J$ can be characterized
by
$$J = \overline{ \{x \in M(C) : \tau(x^*x) < \infty \}}.$$
(See \cite{RorMult}; see also \cite{Elliott1974}, \cite{LinMult1988},
\cite{Elliott1987}, and \cite{ZhRiesz}.)}
\end{NN}

\begin{df}\label{Dconscale}

{\rm
Retaining the notation from 2.7, recall that a C*-algebra
$C$ is said to have {\it continuous scale}  if
for all $x \in C_+ \setminus \{ 0 \}$, there exists $n_0$ such that for all
$m > n \geq n_0$, $(e_m - e_n) \preceq x$. (\cite{LinContScale} Definition 2.5.)
(Note that this implies that $M(C)=J_0 = J$.)
If $C$ has continuous scale and $T(C)\not=\emptyset,$ 
then  the map $\tau\to \sup\{\tau(e_n): n\in \N\}$ is a continuous function. 
Moreover, if, in addition,  $C$ is not stable,$d_{\mu}(a) < \infty$ for all $\mu \in T(C)$. 
% $T(eCe)$.
% where
%$e\in C$ is a non-zero projection of $C.$

Let $A$ be a unital simple \CA\, and let $C\subset A\otimes {\cal K}$  be a non-unital
hereditary \SCA.
Suppose that $C$ has continuous scale. %(see \cite{LinContScale}
%for the definition).
Suppose also that $T(A)\not=\emptyset.$ Then
$$
f(t)=\sup\{t(a): a\in C\andeqn 0\le a\le 1\}
$$
(for $t\in T(A)$) is a continuous affine function in $\Aff(T(A)).$
Since $A$ is simple,
$$
\inf\{f(t): t\in T(A)\}>0.
$$
For each $t\in T(A),$ define
$$
\tau(c)=t(c)/f(t)\tforal c\in C.
$$
Then $\tau$ is a normalized trace on $C.$ Note that every $\tau\in T(C)$ has this form. This also implies that $T(C)$ is compact.

Let $C$ be a non-unital and $\sigma$-unital nonelementary simple \CA.
Let $J\subseteq M(C)$ be the ideal as defined in \ref{MinIdeal} and let
$a\in J_+\setminus \{0\}.$ Then for all $\delta > 0$,
$C_1=\overline{(a-\delta)_+C(a - \delta)_+}$ has continuous scale.
}

\end{df}

\section{Exponential length}

The following is known.
\begin{lem}\label{QT} {\rm (cf. 10.8 of \cite{Lnappeqv})}
Let $\ep > 0$ and  let $\Delta: (0, 1) \to  (0, 1)$ be a non-decreasing map. There exist $\eta>0,\, \dt > 0$ and
a finite subset ${\cal G}\in C(\T)_{s.a.}$  satisfying
the following:
Suppose that $A$ is a unital separable simple \CA\, with $TR(A)=0,$
and suppose that $u, v\in U(A)$ are two unitaries such that ${\rm sp}(u)={\rm sp}(v)=\T,$
\beq\label{Appu-1}
\mu_{\tau\circ \phi}(I_a)\ge \Delta(a)
  \tforal \tau\in T(A)
 \eneq
 and for all arcs $I_a$ with length at least $a\ge \eta,$
 where $\phi: C(\T)\to A$ is defined by $\phi(f)=f(u)$ for all $f\in C(\T)$ and
 \beq\label{Appu-2}
&& |\tau(g(u))-\tau(g(v))|<\dt\tforal g\in {\cal G},
\tforal \tau\in T(A),\\\nonumber
&&{[u]}={[v]}\,\,\, {\rm in} \,\,\, K_1(A),
\eneq
Then there exists a unitary  $w\in U(A)$ such that
\beq\label{Appu-3}
\|w^*uw-v\|<\ep.
\eneq
\end{lem}

The following is a non-unital version of the above.

\begin{lem}\label{Appu}
Let $\ep > 0$ and  let $\Delta: (0, 1) \to  (0, 1)$ be a non-decreasing map. There exist $\eta>0,\, \dt > 0$ and
a finite subset ${\cal G}\in C(\T)_{s.a.}$  satisfying
the following:
Suppose that $A$ is a unital separable simple \CA\, with $TR(A)=0,$
$C$ is a non-unital hereditary \SCA\, of $A\otimes {\cal K}$ with continuous scale,
 $B={\tilde C}$
and suppose that $u, v\in U(B)$ are two unitaries such that ${\rm sp}(u)={\rm sp}(v)=\T,$
\beq\label{Appu-1}
\mu_{\tau\circ \phi}(I_a)\ge \Delta(a)
  \tforal \tau\in T(C)
 \eneq
 and for all arcs $I_a$ with length at least $a\ge \eta,$
 where $\phi: C(\T)\to B$ is defined by $\phi(f)=f(u)$ for all $f\in C(\T)$ and
 \beq\label{Appu-2}
&& |\tau(g(u))-\tau(g(v))|<\dt\tforal g \in {\cal G} \andeqn\,\,\,
\tforal \tau\in T(C),\\
&&{[u]}={[v]}\,\,\, {\rm in} \,\,\, K_1(B)\andeqn \pi(u)=\pi(v),
\eneq
where $\pi: B\to B/C=\C$ is the quotient  map.
Then there exists a unitary  $w\in U(B)$ such that
\beq\label{Appu-3}
\|w^*uw-v\|<\ep.
\eneq
\end{lem}

\begin{proof}
Without loss of generality, we may assume
that $\pi(u)=\pi(v)=1.$
Let $\ep>0$ and let $\{e_n\}$ be an approximate identity for $C$ consisting of projections.
Let $\Delta_1(a)=(1/3)\Delta(3a/4)$ for $a\in (0,1)$ and
$\Delta_2=\Delta_1/2.$

Let $\dt>0$ and $\eta_1>0$ (in place of $\eta$) and a 
finite subset ${\cal G}\subset C(\T)$ be required by Lemma \ref{QT} for  $\ep/2$ and
$\Delta_2.$
Put $\eta=\eta_1/4.$
Suppose that $u$ and $v$ are two unitaries in $B$ satisfy the assumption for the above
$\dt$ and $\eta.$

Let $\ep/4>\ep_0>0.$
Since $C$ has continuous scale, there exists $N\ge 1$ and unitaries $u_1,\, v_1\in e_NAe_N=e_NCe_N$ such that
$$
\|u-((1-e_N)+u_1)\|<\ep_0\andeqn \|v-((1-e_N)+v_1)\|<\ep_0,
$$
and
\beq\label{exL1-0}
\tau(1-e_N)<\min\{\dt/16, \Delta(\eta)/16\}
\eneq
for all $\tau \in T(C)$.

By choosing sufficiently small $\ep_0,$
we may assume that
\beq\label{exL1-1}
|\tau(g(u_1))-\tau(g(v_1))|<\dt\tforal \tau\in T(e_NAe_N)
\eneq
and for all $g\in {\cal G}.$
Moreover (by Lemma 3.4 of \cite{LnHAH}),
\beq\label{exL1-2}
\mu_{\tau\circ \psi}(I_a)\ge \Delta_1(a)\tforal \tau\in T(C)
\eneq
and for all arcs $I_a$ with length $a\ge \eta,$  where
$\psi(f)=f((1-e_N)+u_1)$ for all $f\in C(\T).$
It follows from (\ref{exL1-0}) and (\ref{exL1-2}) that
$$
\mu_{\tau\circ \psi_1}(I_a)\ge \Delta_2(a)\tforal \tau\in T(e_NAe_N)
$$
and for all $a\ge \eta,$ where $\psi_1(f)=f(u_1)$ for all $f\in C(\T).$
It follows from \ref{QT} that there exists a unitary $w_1\in e_NAe_N$ such that
$$
\|u_1-w_1^*v_1w_1\|<\ep/2.
$$
Let $w=(1-e_N)+w_1.$ Then
$$
\|u-w^*vw\|<\ep.
$$

\end{proof}

\begin{prop}\label{Delta}
Let  $A$ be a unital simple \CA,
%\, with $TR(A)=0,$
$C\subset A\otimes {\cal K}$ be a non-unital hereditary \SCA\, with continuous scale and let $B={\tilde C}.$ Let $u\in U(B)$ with
${\rm sp}(u)=\T.$
Then there exists a nondecreasing function $\Delta: (0,1)\to (0,1)$ such that
\beq\label{Delta-1}
\mu_{\tau\circ \phi}(I_a)\ge \Delta(a)\tforal \tau\in T(C)
\eneq
and for all arcs $I_a$ with length $a\in (0,1),$  where $\phi: C(\T)\to B$ defined by $\phi(f)=f(u)$ for all $f\in C(\T).$
Moreover, one may also require that $\lim_{a\to 0+} \Delta(a)=0.$
\end{prop}

\begin{proof}
Let $\pi: B\to B/C=\C$ be the quotient map. Let $\lambda=\pi(u).$ So $\lambda\in \T.$
Fix $a\in (0,1).$
Consider finitely many open arcs $I_{a,1}, I_{a, 2},...,I_{a, m}\subset \T$ all
with radius $a/2$ such that
$\lambda\in I_{a,1}$ and every arc $I_a$ with length at least $a$ contains one of  $I_{a, j}$ and
$\cup_{j=1}^m I_{a, j}=\T.$
Choose a non-zero positive function $f_{a, j}\in C(\T)$ such that the support of $f_{a,j}$ contained in $I_{a,j}$
and $1\ge f_{a, j}(t)>0$ for all $t\in I_{a, j},$
$j=1, 2,3,...,m.$  Since $f_{a,j}\in C$ for $j=2,3,...,m,$ and $T(C)$
is compact (since $C$ has continuous scale),
$$
d_{a, j}=\inf\{\tau(f_{a,j}): \tau\in T(C)\}>0.
$$
For $f_{a,1},$ it dominates a non-zero positive element in $C.$ It follows that
$$
d_{a, 1}=\inf\{\tau(f_{a, 1}): \tau\in T(C)\}>0.
$$
Define $D(a)=\min\{ a, \, d_{a, j}: j=1,2,...,m\}.$
Then define
$$
\Delta(a)=\sup \{D(\eta): 0<\eta\le a\}.
$$
Note that, for each $a\in (0,1),$ $I_{a}\supset I_{b,j}$ for some $j$ and any $b\le a.$  Therefore
$$
\mu_{\tau\circ \phi}(I_a)\ge \mu_{\tau\circ \phi}(I_{b, j})
$$
for all $\tau \in T(C)$.
It follows that
$$
\mu_{\tau\circ \phi}(I_a)\ge D(b)\tforal \tau\in T(C)\andeqn \tforal a \ge b.
$$
Consequently
$$
\mu_{\tau\circ \phi}(I_a)\ge \Delta(a)\tforal \tau\in T(C).
$$
%for all $\tau \in T(C)$.
\end{proof}

\begin{lem}\label{expR0}
Let  $A$ be a unital  simple separable \CA\, with $TR(A)=0,$
$C\subset A\otimes {\cal K}$ be a non-unital hereditary
\SCA\, with continuous scale and let $B={\tilde C}.$ Let $u\in U_0(B)$ with
${\rm sp}(u)=\T$ and $\pi(u)=1,$ where
$\pi: B\to \C$ is the quotient map.   Then, for any $\ep>0,$
there exists a selfadjoint element $h\in C$ with ${\rm sp}(h)=[-2\pi, 2\pi]$ such that
\beq\label{expR0-1}
\|u-\exp(ih)\|<\ep\andeqn \tau(h)=0\tforal \tau\in T(C).
\eneq
Moreover, we may assume that
\beq\label{expR0-1+1}
\sup_{\tau\in T(C)} \lim_{n\to\infty} \tau(|h|^{1/n})<1.
\eneq
\end{lem}

\begin{proof}
First we note that the assumption that $sp(u)=\T$ implies that $A$ is infinite dimensional.
Let $\ep>0.$  Let $\phi: C(\T)\to B$ be the \hm\, defined by $\phi(f)=f(u)$ for all $f\in C(\T).$
It follows from \ref{Delta} that there exists a non-decreasing function $\Delta: (0,1)\to (0,1)$ such that
$$
\mu_{\tau\circ \phi}(I_a)\ge\Delta(a)\tforal \tau\in T(C)
$$
and for all arcs $I_a$ with length $a\in (0,1).$
%Define $\Delta=\Delta_1/2.$
Let $\eta>0,$ $\dt>0$ and ${\cal G}\subset C(\T)_{s.a.}$ a finite subset be required by \ref{Appu} for
$\ep/4$ and $\Delta.$ To simplify the notation, without loss of generality, we may assume that $\|g\|\le 1$ for all $g\in {\cal G}.$

Let $\ep/4>\ep_0>0.$
Let $\{e_n\}$ be an approximate identity  for $C$ consisting of projections.
We choose an integer $N\ge 1$ and unitary $u_1\in e_NCe_N$ such that $u_1\in U_0(e_NCe_N)$ and
\beq\label{expR0-2}
\|u-((1-e_N)+u_1)\|<\ep_0/2\andeqn \tau(1-e_N)<\min\{\ep_0, \Delta(\eta/2)/4\}
\eneq
for all $\tau\in T(C).$
Since $e_NCe_N$ has real rank zero and is infinite dimensional,    $u_1\in CU(e_NCe_N).$
It follows from Theorem 4.5 of \cite{Lnexp} that there
exists a selfadjoint element $b_1\in e_NCe_N$ with $\|b_1\|\le 2\pi$
such that
\beq\label{expR0-3}
\|u_1-e_N\exp(i b_1)\|<\ep_0/2 \andeqn \tau(b_1)=0\tforal \tau\in T(e_NCe_N).
\eneq
(This can be derived directly from the fact that $C$ has real rank zero).
%This certainly also holds when $e_NCe_N$ is finite dimensional.
We assume that $e_{N+1}-e_N\not=0.$
Let $q_1, q_2\in (e_{N+1}-e_N)C(e_{N+1}-e_N)$ be mutually orthogonal and mutually equivalent projections such that
\beq\label{expR0-3+}
|\tau(q_2)|=|\tau(q_1)|<\min\{\ep_0/32\pi, \Delta_1(\eta/2)/32\pi\}
\tforal \tau\in T(C).
\eneq
%there exists a selfadjoint element $b_1\in e_NCe_N$ with $\|h_1\|\le \pi$
%such that
%\beq\label{expR0-3}
%\|u_1-\exp(i b_1)\|<\ep_0/2.
%\eneq
%Since $TR(A)=0,$ ${\overline{\rho_A(K_0(A))}}=\aff(T(A)).$ Since $A$ is simple and $C$ is a  hereditary,
%there exists $a\in \aff (T(C))$ such that $\|a\|<\min\{\ep_0/8, \Delta(\eta/2)/4\}$ such that
%\beq\label{\expR0-4}
%a-\hat{b_1}\in \rho_C(K_0(C)).
%\eneq
%Let $x\in K_0(C)$ such that $\tau(x)=a(\tau)-\tau(b_1)$ for all $\tau\in T(C).$
%Choose a projection $p\in M_n(C)$ for some $n\ge 1$ such that  $  (1+\ep_0/8)\rho_C(x)> \rho_C(p)>\rho_C(x)$ in
%$\aff(T(C)).$ This is possible since $A$ has real rank zero and simple. Hence $\rho_C(p)-\rho_C(x)>0.$
%It follows that there is a projection $q\in M_k(C)$ for some $k\ge 1$ such that
%$[q]=[p]-x.$ It follows that $p-q=x.$ Note that $

We choose $\ep_0$ sufficiently small such that $\ep_0<\dt/16\pi$ and
\beq\label{expR0-4}
|\tau(g(u))-\tau(g(u_1'))|<\dt/2\tforal g\in {\cal G}
%\andeqn
\eneq
%whenever $\|u-u_1'\|<\ep_0$ for any unitary $v_1'.$
%\beq\label{expR0-5}
%\mu_{\tau\circ \psi}(I_a)\ge  (2/3) \Delta_1(a)\tforal \tau\in T(C)\andeqn
%\eneq
%for all arcs $I_a$ with length $a\ge \eta,$
where $u_1'=(1-e_N)+u_1.$
%and $\psi: C(\T)\to B$ is the \hm\, defined by $\psi(f)=f(u_1')$ for all $f\in C(\T).$
Let $b_2\in q_1Cq_1$ with ${\rm sp}(b_2)=[-2\pi, 2\pi].$  Let $z\in U(B)$ such that $z^*q_1z=q_2.$
Let $b_3=-z^*b_2z.$ Note that $\tau(b_2+b_3)=0$ for all $\tau\in T(C).$
Define $u_2=(1-e_{N+1})+(e_{N+1} - e_N
-q_1-q_2)+ (q_1+q_2)\exp(i (b_2+b_3)) +u_1.$
By  (\ref{expR0-4}), (\ref{expR0-3+})  and the fact that $\|g\|\le 1$ for all $g\in {\cal G},$
% and (\ref{expR0-5}),
we estimate that
\beq\label{expR0-6}
|\tau(g(u))-\tau(g(u_2))|<\dt\tforal g\in {\cal G}.
\eneq
%\beq\label{expR0-7}
%\mu_{\tau\circ \psi_1}(I_a)\ge  \Delta(a)\tforal \tau\in T(C)\andeqn
%\eneq
%for all arcs $I_a$ with $a\ge \eta,$ where $\psi_1: C(\T)\to B$ is the \hm\, defined by $\psi_1(f)=f(u_2)$ for all $f\in C(\T).$
It follows from \ref{Appu} that there exists a unitary $w\in B$ such that
\beq\label{expR0-6}
\|u-w^*u_2w\|<\ep/2.
\eneq

Now let $h=w^*(b_1+b_2+b_3)w.$  Then  ${\rm sp}(h)=[-2\pi, 2\pi],$
$$
\|u-\exp(i h)\|<\ep\andeqn \tau(h)=0\tforal \tau\in T(C).
$$
Moreover (\ref{expR0-1+1}) also holds since $h\in (e_N+q_1+q_2)C(e_N+q_1+q_2).$
\end{proof}

\begin{cor}\label{fullsp}
Let $A$ be a unital simple infinite dimensional separable \CA\, of tracial rank zero and let
$C\subset A\otimes {\cal K}$ be a non-unital hereditary \SCA\, with continuous scale. Suppose that $u\in U({\tilde{C}})$ with ${\rm sp}(u)=\T$ and
suppose that $\{e_n \}\subset C$ is an approximate identity consisting of projections.
Then, for any $\ep>0$ and any $\sigma>0,$  there exists $k\ge 1$ and a
unitary $w\in e_kCe_k$ with ${\rm sp}(w)=\T$  such that
\beq\label{fullsp-1}
\|u-\lambda(1-e_k+w)\|<\ep\andeqn \tau(1-e_k)<\sigma\tforal \tau\in T(C),
\eneq
{where $\lambda=\pi(u)$ and $\pi: {\tilde C}\to {\tilde C}/C$ is the quotient map.}
\end{cor}

\begin{proof}
Without  loss of generality, we may assume that $\pi(u)=1.$
In the above proof, let
$$w=(e_{N+1}-e_N-q_1-q_2)+(q_1+q_2)\exp(i(b_2+b_3)+u_1$$
Then, since $sp(b_2)=[-2\pi, 2\pi],$ ${\rm sp}(w)=\T.$
Let $k=N+1.$ The corollary then follows.
\end{proof}

The following is a non-unital version of Corollary 3.9 of \cite{Lnmem}.

\begin{lem}\label{homt}
Let $\ep>0.$ There exists $\dt>0$ satisfying the following:
For any unital separable simple \CA\,  $A$ with real rank zero and stable rank one,
any hereditary \SCA\, $C\subset A\otimes {\cal K}$ with continuous scale, and any unitary
$u\in  {\tilde C}$ with ${\rm sp}(u)=\T,$    if $v\in {\tilde C}$ is another unitary with $[v]=0$ in $K_1({\tilde C})$
such that
\beq\label{homt-1}
\|[u, \, v]\|<\dt\andeqn {\rm bott}_1(u,v)=0,
\eneq
there exists a continuous path of unitaries $\{V(t): t\in [0,1]\}$ in
$\tilde{C}$  with $V(0)=v$ and $V(1)=1$ such that
\beq\label{homt-2}
\|[u, \, V(t)]\|<\ep\tforal t\in [0,1]\andeqn {\rm length}(V(t))\le \pi +\ep.
\eneq
Moreover, if $\pi(v)=1, $ one can require that  $\pi(V(t))=1$ for all $t\in [0,1].$
\end{lem}

\begin{proof}
Without loss of generality, we may assume that
$\pi(u)=1.$  Put ${\tilde C}=B.$ Let $1/2>\ep>0.$
From Corollary 3.9 of \cite{Lnmem}, one has the following statement:
There exists $\dt>0$  and $\sigma>0$ satisfying the following: for any unital separable simple \CA\, $A_0$ with real rank zero and stable  rank one,
and unitary $u'\in A_0$ with ${\rm sp}(u')$  being $\sigma$-dense in $\T$ and any unitary $v'\in A_0$ with $[v']=0$ in $K_1(A_0)$ such that
$\|[u,\, v']\|<\dt$ and ${\rm bott}_1(u', v')=0,$ there exists a continuous path of unitaries $\{v'(t): t\in [0,1]\}\subset A_0$
such that $v'(0)=v',$ $v'(1)=1$ and $\|[u',\, v'(t)]\|<\ep/4$ for all $t\in [0,1]$ and
${\rm length}(\{v'(t)\}: t\in [0,1]\}\le \pi+\ep/2.$
Choose such $\dt$ and $\sigma.$

Choose $\dt_1>0$ satisfying the following: if $u'$ and $v'$ are two unitaries with ${\rm sp}(u')=\T,$ then
${\rm sp}(v')$ is $\sigma/2$-dense in $\T,$ provided that $\|u'-v'\|<\dt_1.$

%Let $1/2>\ep>0.$ Let $\dt>0$ be given by 3.4 of \cite{Lnexp} (see also \cite{Lnmem})  for $\ep/4$ (in place of $\ep$).
Let $\{e_n\}$ be an approximate identity for $C$ consisting of projections.
Choose $\theta>0$ satisfying the following:
for any unitaries $w_1, w_2, w_3,$
${\rm bott}_1(w_1, w_3)$ and ${\rm bott}_1(w_2, w_3)$ are well defined and
$$
{\rm bott}_1(w_1, w_3)={\rm bott}_1(w_2, w_3)
$$
provided that $\|[w_j, w_3]\|<\theta$ and $\|w_1-w_2\|<\theta.$

Choose $\ep_1=\min\{\ep/4, \dt/4, \dt_1/2, \theta\}.$
Choose an integer $N\ge 1$ and unitaries
 %$u_0,\, v_0\in (1-e_N)B(1-e_N),$
 $u_1, v_1\in e_NAe_N=e_NCe_N$ such that
\beq\label{homt-3}
\|u-((1-e_N)+u_1)\|<\ep_1,\,\,\, \|v-(\lambda(1-e_N)+v_1)\|<\ep_1
%\|u_0-(1-e_N)\|<\ep_1,\,\,\,\|v_0-\lambda(1-e_N)\|<\ep_1\\
\andeqn
\|[u_1,\, v_1]\|<\ep_1
\eneq
where $\lambda=\pi(v).$   Moreover,  by the choice of $\theta,$
\beq\label{homt-4}
{\rm bott}_1(u,v)={\rm bott}_1(u_1', v_1'),
\eneq
where
$$
u_1'=(1-e_N)+u_1\andeqn v_1'=\lambda(1-e_N)+v_1.
$$
Note that
\beq\label{homt-5}
{\rm bott}_1(u_1', v_1')={\rm bott}_1(u_1, v_1).
\eneq
It follows from (\ref{homt-4}), and (\ref{homt-5})
\beq\label{homt-7}
{\rm bott}_1(u_1, v_1)=0.
\eneq
By the choice of $\dt_1,$ we conclude that ${\rm sp}(u_1')$ is $\sigma/2$-dense in $\T.$ Thus
${\rm sp}(u_1)$ is $\sigma$-dense in $\T.$
By applying the statement at the beginning of this proof (from Cor.3.9 of \cite{Lnmem}),   we obtain  a continuous path of unitaries
$\{w(t): t\in [1/2,1]\}\subset e_NAe_N$ such that
\beq\label{homt-8}
w(1/2)=v_1,\,\,\, w(1)=e_N\andeqn
\|[u_1,\, w(t)]\|<\ep/4\tforal t\in [1/2,1].
\eneq
Moreover,
\beq\label{homt-9}
{\rm length}\{w(t): t\in [1/2, 1]\}\le \pi+\ep/2.
\eneq
By (\ref{homt-3}), there exists $h\in B_{s.a.}$ such that
\beq\label{homt-10}
\|h \|<2\arcsin (\ep/4)\andeqn v\exp(ih)=v_1'.
\eneq
Define $w_0(t)=v\exp(i2th)$ for $t\in [0,1/2].$
Then
$$w_0(0)=v,\,\,\,w_0(1/2)=v_1'
\andeqn {\rm length}\{w_0(t): t\in [0,1/2]\}<\ep/2.
$$
Now define $V(t)=w_0(t)$ if $t\in [0,1/2)$ and
$V(t)=\lambda(1-e_N)+w(t)$ for $t\in [1/2, 1].$
Then $V(t)$ is continuous and
$$
{\rm length}\{V(t)\}\le \pi+\ep/2+\ep/2=\pi+\ep.
$$
One also verifies that
\beq\label{homt-11}
\|[u, \, V(t)]\|<\ep\tforal t\in [0,1].
\eneq

For the very last part of the lemma, assume that $\pi(v)=1.$ Then $\lambda=1.$
By (\ref{homt-10}), $\pi(h)=0.$ It follows that $\pi(w_0(t))=1$ for all $t\in [0,1/2].$
One then checks that
$$
\pi(V(t))=1\tforal t\in [0,1].
$$

\end{proof}

\begin{lem}\label{Ruv}
Let $C$ be a non-unital hereditary \SCA\, of a separable \CA\,
with stable rank one, and let $B = \tilde{C}$.
 %$A\otimes {\cal K},$ where $A$ is a unital separable simple \CA\,  with stable rank one and let $B={\tilde C}.$
Suppose that $u$ and $v$ are two unitaries in $B$ such that $uv^*\in U_0(B)$ and
$\pi(u)=\pi(v),$ where $\pi: B\to B/C=\C$ be the quotient map.
%Suppose also that
%\beq\label{Ruv-1}
%\sup_{\tau\in T(A)}\sup\{\tau(c): 0\le a\le 1\,\,\, c\in C\}<\infty.
%\eneq
Then, we can always assume that
$$
R_{u,v}(t_0)=0
$$
where $t_0$ is the tracial state of $B$ such that $(t_0)|_C=0.$
%Therefore, we may assume that
%$$
%R_{u,v}\in \aff(T(C)).
%$$
\end{lem}

\begin{proof}
We may write
$$
uv^*=\prod_{j=1}^k \exp (\pi i a_j)
$$
where $a_j\in B_{s.a.}.$
Since $\pi(uv^*)=1,$
$$
\sum_{j=1}^k \pi(a_j)=2m
$$
for some integer $m.$
Define $h_1=a_1-2 m,$ $h_j=a_j,$ $j=2,3,...,k.$
Define
$$
U(t)=\prod_{j=1}^k \exp(\pi i (1-t)h_j)v\,\,\,\,t\in [0,1].
$$
Then $U(t)$ is a continuous piecewise smooth path with $U(0)=u$ and $U(1)=v.$
Moreover,
$$
\sum_{j=1}^k \pi(h_j)=0.
$$
It follows that
$$
{1\over{2\pi i}}\int_0^1 t_0({d U(t)\over{dt}}U(t)^*) dt=0.
$$

\end{proof}

The following is a non-unital version for a special case of Lemma 3.5 of \cite{Lnexp}.

\begin{lem}\label{LN51}
Let $\ep>0$ and let $\Delta: (0,1)\to (0,1)$ be a non-decreasing function.
There is $\dt>0,$ $\eta>0,$ $\sigma>0$ and there is a finite subset
${\cal G}\subset C(\T)_{s.a.}$ satisfying the following:
For any unital separable simple \CA\, $A$ with $TR(A)=0,$
any non-unital hereditary \SCA\, $C\subset A\otimes {\cal K}$ with continuous scale,  any pair of unitaries $u, v\in {\tilde C}$ such that $sp(u)=\T$ and
$[u]=[v]$ in $K_1(\tilde{C}),$ $\pi(u)=\pi(v),$
$$
\mu_{\tau\circ \phi}(I_a)\ge \Delta(a)\tforal \tau\in T(C)
$$
for all intervals $I_a$ with length at least $\eta,$ where
$\phi: C(\T)\to \tilde{C}$
is the \hm\, defined by $\phi(f)=f(u)$ for all $f\in C(\T),$
\beq\label{LN51-2}
|\tau(g(u))-\tau(g(v))|<\dt \tforal \tau\in T(C)
\eneq
and for all $g\in {\cal G},$
%\beq\label{LN51-3}
%%uv^*\in CU(\tilde{C}),
%\eneq
for any $a\in Aff(T(C))$ with $a-R_{u,v}|_{T(C)}\in \rho_C(K_0(C))$ and $\|a\|<\sigma$  and $y\in K_1({\tilde C}),$ there is a unitary
$w\in \tilde{C}$ such that
\beq\label{LN51-4}
[w]=y,\,\,\,
\|u-w^*vw\|<\ep\andeqn\\
{1\over{2\pi i}}\tau(\log(u^*w^*vw))=a(\tau)\tforal \tau\in T(C).
\eneq
\end{lem}

\begin{proof}
Note that, since ${\rm sp}(u)=\T,$ $A$ has infinite dimension. Since $TR(A)=0,$ $CU(A)=U_0(A)$ and
$\overline{\rho_A(K_0(A)))}=\Aff(T(A)).$ We will use these facts in the following proof.
Without loss of generality, we may assume that
$\pi(u)=\pi(v)=1.$  By \ref{Ruv}, we may assume that
\beq\label{LN51-4+}
R_{u,v}(t_0)=0,
\eneq
where $t_0$ is the tracial state of $B$ that vanishes on $C.$

Let $\ep>0$ and $\Delta$ be  given.
%Let $\Delta_1=\Delta/2.$
Choose $\ep/2>\theta>0$ such that,
$\log(u_1),$ $\log(u_2)$ and $\log(u_1u_2)$ are well defined and
\beq\label{LN51-n1}
\tau(\log(u_1u_2))=
\tau(\log(u_1))+ \tau(\log(u_2))
\eneq
for every tracial state $\tau$ and
for any unitaries
$u_1, u_2$
such  that
$$
\|u_j-1\|<\theta,\,\,\, j=1,2.
$$
We may choose even smaller $\theta$ such that
\beq\label{LN51-n2}
{\rm bott}_1(u_1, v_1)={\rm bott}_1(u_2, v_1)
\eneq
provided that $\|[u_1,v_1]\|<\theta\andeqn \|u_1-u_2\|<\theta.$
%need two corrections in JOT paper
Let $\dt'>0$ (in place of $\dt$) be required by Lemma 3.1 of \cite{Lnexp}  for $\theta/2$ (in place of $\ep$).
Put $\sigma=\dt'/2.$
Let $1/2>\dt>0$ and $\eta$ be required by \ref{Appu} for $\min\{\sigma, \theta/2, 1\}$ (in place of $\ep$) and $\Delta.$
%Without loss of generality, we may assume that $\eta<\sigma/16.$
Suppose  $u$ and $v$ satisfy the
 assumption for the above $\dt,$ $\eta$ and
$\sigma.$
Then, by \ref{Appu}, there exists a unitary $z\in U(\tilde{C})$ such that
\beq\label{LN51-5}
\|u-z^*vz\|<\min\{\sigma, \theta/2, 1\}.
\eneq
Let $b={1\over{2\pi i}}\log(u^*z^*vz).$ Then
$\|b\|<\min\{\theta, \sigma, 1\}.$  By Lemma 3.2 of \cite{Lnexp},
$\hat{b}-R_{u,v}\in \rho_{C}(K_0(C)).$  Note that, by (\ref{LN51-5}),
$\|\log(u^*z^*vz)\|<\pi.$ Since $\pi(u)=\pi(v)=1,$  $\pi(u^*z^*vz)=1.$ It follows that
$\pi(\log(u^*z^*vz))=0.$   In particular, $t_0(b)=0.$

Let $a\in Aff(T(C))$ be such that
$\|a\|<\sigma$ and $a-R_{u,v}|_{T(C)}\in \rho_C(K_0(C))$ as given by the lemma.
It follows that $a-\hat{b}\in \rho_C(K_0(C)).$ Moreover, $\|a-\hat{b}\|<2\sigma<\dt'.$
Let $p, q\in M_k(C)$ be projections such that
\beq\label{LN51-5+1}
\tau(p-q)=a(\tau)-\tau(b)\tforal \tau\in T(C).
\eneq
%Choose $\eta=\sigma/16.$
Let $\{e_n\}$ be an approximate identity for $C$ consisting of projections. Choose an integer $N\ge 1$  and unitaries
%$u_0, v_0\in (1-e_N)C(1-e_N),$
$u_1, v_1\in e_NAe_N=e_NCe_N$ such that
\beq\label{LN51-5+2}
&&\hspace{-0.8in}\|u-((1-e_N)+u_1)\|<\min\{\dt/4, \theta/2\},\,\,\,\|v-((1-e_N)+v_1)\|<
\min\{\dt/4,\theta/2\}\\\label{LN51-5+3}
&&\andeqn \tau(1-e_N)<\eta\tforal  \tau\in T(C).
\eneq
By \ref{fullsp}, we may also assume that ${\rm sp}(u_1)=\T.$
By (\ref{LN51-5+1}) and (\ref{LN51-5+3}),
\beq\label{LN51-5+4}
|\tau(p-q)|<\sigma\tforal \tau\in T(e_NAe_N).
\eneq
Note that $K_1({\tilde C})=K_1(C)=K_1(eNAe_N).$
It follows from Lemma 3.1 of \cite{Lnexp}  that there exists a unitary $z_1\in e_NAe_N$ such that
\beq\label{LN51-6}
[z_1]=y-[z],\,\,\, \|[u_1, z_1]\|<\theta/2\andeqn {\rm bott}_1(u_1,z_1)(\tau)=\tau(p-q)
\eneq
for all $\tau\in T(e_NAe_N).$
Put $z_2=(1-e_N)+z_1,$ $w=zz_2$ and $u_1'=(1-e_N)+u_1.$
It follows that
\beq\label{LN51-6+0}
[w]=y\andeqn \|u-w^*vw\|<\theta<\ep
\eneq
\beq\label{LN51-6+1}
{\rm bott}_1(u_1',z_2)(\tau)=a(\tau)-\tau(b)\tforal \tau\in T(C).
\eneq
%Let $w_1=(1-e_N)+z_1.$
Put $u_0=u((1-e_N)+u_1^*)$ and $v_0=v({\bar \lambda}(1-e_N)+v_1^*).$
So
\beq\label{LN51-7}
\|u_0-1\|<\min\{\dt/4, \theta/2\}\andeqn \|v_0-1\|<\min\{\dt/4,\theta/2\}.
\eneq
By the choice of $\theta,$ we have
\beq\label{LN51-7+}
{\rm bott}_1(u, z_2)={\rm bott}_1(u_1',z_2).
\eneq
We compute that (using (\ref{LN51-n1}), (\ref{LN51-n2}), (\ref{LN51-7+}) and the Exel formula (see Lemma 3.5 of \cite{Lninv}))
\beq\label{LN51-8}
{1\over{2\pi i}}\tau(\log(u^*w^*vw))&=&
{1\over{2\pi i}}\tau(\log(u^*z_2^*z^*vzz_2))\\
&=&{1\over{2\pi i}}\tau(\log(u^* z_2^* (u  u^*) z^* v z z_2)\\
&=& {1\over{2\pi i}}\tau(\log((z_2u^*z_2^*u)( u^*z^*vz))\\
&=&{1\over{2\pi i}}(\tau(\log(z_2u^*z_2^*u))+\tau(\log(u^*z^*vz)))\\
&=&{1\over{2\pi i}}(\tau(\log(u^*z_2^*uz_2))+\tau(b)\\
&=&{\rm bott}_1(u,z_2)(\tau)+\tau(b)\\
&=& a(\tau)-\tau(b)+\tau(b)=a(\tau)\,\,\,\rforal \tau\in T(C).
\eneq
\end{proof}

The above proof also contains the following

\begin{lem}\label{botext} ({\rm see 3.1 of \cite{Lnexp}})
Let $\ep>0.$ There exists $\dt>0$ satisfying the following:
Suppose that $A$ is a unital separable simple \CA\, with $TR(A)=0,$ $C\subset A\otimes {\cal K}$ is a non-unital hereditary
\SCA\, with continuous scale and suppose
that $u\in U({\tilde{C}})$ with ${\rm sp}(u)=\T.$ Then, for any
$x\in K_0(C)$  with $\|\rho_C(x)\|<\dt$ and any $y\in K_1(C),$  there exists a unitary
$v\in A$ such that
\beq\label{botext-1}
[v]=y,\,\,\,\|[u,\, v]\|<\ep\andeqn {\rm bott}_1(u,v)=x.
\eneq
\end{lem}

The next result is a special case of \cite{TomsWinter} Corollary 3.1.
We provide a short proof for the convenience of the reader.

\begin{lem}\label{herd}
Let $A$ be a unital separable simple \CA\, which is ${\cal Z}$-stable and
let $C\subset A\otimes {\cal K}$ be a hereditary \SCA\
with continuous scale.
Then $C\otimes {\cal Z}\cong C.$
\end{lem}

\begin{proof}
Let $C=\overline{a(A\otimes {\cal K})a}$ for some positive element $a\in A\otimes {\cal K}.$ By the assumption, we may assume
that
$$
a\lesssim E_N,
$$
where $E_N={\rm id }_{M_N(A)}.$
Since $A$ is ${\cal Z}$-stable, $A$ has stable rank one. We may assume that $a\in M_N(A).$
So $C\otimes {\cal Z}$ is a hereditary \SCA\, of $M_N(A)\otimes {\cal Z}\cong M_N(A).$
There is an isomorphism $\phi: M_n(A)\otimes {\cal Z}\to M_n(A)$ 
such that $\phi(c\otimes 1_{\cal Z})$ is approximately unitarily equivalent  to $c$ for all $c\in M_n(A).$ 
%There is a sequence of unitaries $\{u_n\}\in M_n(A)\otimes {\cal Z}$ such that
Let $B_1=\overline{b(A\otimes M_N)b}=C\otimes {\cal Z},$ where $b=a\otimes 1_{\cal Z}.$  Then
$$
[a]=[\phi(b)]
$$
in the Cuntz semigroup.
%$$
%d_{\tau\otimes t_0}(b)=d_{\tau}(a)\tforal \tau\in T(A),
%$$
%where $t_0$ is the unique tracial state on ${\cal Z}.$ It follows $[a]=[b].$
Therefore $C\otimes {\cal Z}\cong C.$

\end{proof}

\begin{lem}\label{Mapp}
Let $\ep > 0$ and  let $\Delta: (0, 1) \to  (0, 1)$ be a non-decreasing map. There exist $\eta>0,\, \dt > 0$ and
a finite subset ${\cal G}\in C(\T)_{s.a.}$  satisfying
the following:
Suppose that $A$ is a ${\cal Z}$-stable unital separable simple \CA\, such that $TR(A\otimes Q)=0,$
$C$ is a non-unital hereditary \SCA\, of $A\otimes {\cal K}$ with continuous scale,  $B={\tilde C},$
and suppose that $u, v\in U(B)$ are two unitaries such that ${\rm sp}(u)=\T,$
\beq\label{Mp-1}
\mu_{\tau\circ \phi}(I_a)\ge \Delta(a)
  \tforal \tau\in T(C)
 \eneq
 and for all arcs $I_a$ with length at least $a\ge \eta,$
 where $\phi: C(\T)\to B$ is defined by $\phi(f)=f(u)$ for all $f\in C(\T)$ and
 \beq\label{Mp-2}
&& |\tau(g(u))-\tau(g(v))|<\dt\tforal g \in {\cal G},
\tforal \tau\in T(C),\\\nonumber
&&{[u]}={[v]}\,\,\, {\rm in} \,\,\, K_1(C),\,
uv^*\in CU(B)
\andeqn \pi(u)=\pi(v),
\eneq
where $\pi: B\to B/C=\C$ is the quotient map.
Then there exists a unitary  $w\in U({ B\otimes {\cal Z}})$ such that
\beq\label{Mp-3}
\|w^*(u\otimes 1)w-(v\otimes 1)\|<\ep.
\eneq
\end{lem}

\begin{proof}

Without loss of generality, we may assume that $\pi(u)=\pi(v)=1.$
We first note, by \cite{LN-Range},
that $TR(A\otimes M_\fr)=0$ for any supernatural number $\fr.$
Let $\phi: C(\T)\to B$ be the monomorphism defined by
$\phi(f)=f(u).$

It follows from  \ref{Delta}
%Proposition 11.1 of \cite{Lnappeqv}
that there is a non-decreasing function
$\Delta: (0,1)\to (0,1)$ such that
\beq\label{MT1-1}
\mu_{\tau \circ \phi}(O_a)\ge \Delta (a)\tforal \tau\in T(C)
\eneq
for all open balls $O_a$ of $\T$ with radius $a\in (0,1).$
%Define $\Delta=(1/2)\Delta_1.$

%For any $a\in(0, 1)$, denote by
%$$\Delta(a)=\inf\{\mu_{\tau\circ\psi}(O_a);\ \tau\in T(A), \textrm{$O_a$ an open ball of radius $a$ in $X$}\}.$$ Since $A$ is simple and
%$\T$ is compact,  one has that $0<\Delta(a)\leq 1$ (for all $a\in (0,1)$) and $\Delta'(a)\to 0$ as $a\to 0$.

Let $\ep>0.$ As in \ref{herd}, we may assume that $C\subset M_n(A)$ for integer $n\ge 1.$

Let $\p$ and $\q$ be a pair of relatively prime supernatural numbers of infinite type with $\mathbb Q_\p+\mathbb Q_\q=\mathbb Q$. Denote by $M_\p$ and $M_\q$ the UHF-algebras associated to $\p$ and $\q$ respectively.
Let $\imath_\fr: C\to C\otimes M_\fr$ be the embedding defined by
$\imath_\fr(c)=c\otimes 1$ for all $a\in C,$ where $\fr$ is a supernatural number. Note $Tr(M_n(A)\otimes M_\fr)=0.$
Define $u_\fr=\imath_\fr(u)$ and $v_\fr=\imath_\fr(v).$
Denote  $B_\fr={\widetilde{ C\otimes M_\fr}},$ $\fr=\p,\,\q.$ Denote by $\phi_\fr: C(\T)\to B_\fr$ the \hm s defined by $\phi_\fr(f)=f(u_\fr)$
for all $f\in C(\T),$ $\fr=\p,\,\q.$

Let $\dt_1>0$ (in place of $\dt$)
%and $d_1>0$ (in place of $\sigma$)
be required by \ref{homt} for $\ep/6.$
Without loss of generality, we may assume that $\delta_1<\ep/12$  and is small enough
%and $\mathcal G$ is large enough so that for any homomorphism
%$h: C\to A$, the maps
such that
$\mathrm{bott}_1(u_1, z_j)$  and
$\mathrm{bott}_1(u_1, w_j)$ are well defined and
\beq\label{Ap-10}
\mathrm{bott}_1(u_1,w_j)=
\mathrm{bott}_1(u_1, z_1)+\cdots + \mathrm{bott}_1(u_1, z_j)
\eneq
if $u_1$ is a unitary and  $z_j$ is any unitaries with $\|[u_1, u_j]\|<\delta_1,$  where $w_j=z_1\cdots z_j,$ $j=1,2,3,4.$
Let $\dt_2>0$ (in place $\dt$) be require by \ref{botext} for $\dt_1/8$
(in place of $\ep$).

Furthermore, one may assume that $\delta_2$ is sufficiently small such that for any unitaries $z_1, z_2$ in a C*-algebra with tracial states, $\tau(\frac{1}{2\pi i}\log(z_iz_j^*))$ ($i, j=1,2,3$) is well defined and
\vspace{-0.1in}
$$
\tau(\frac{1}{2\pi i}\log(z_1z_2^*))=\tau(\frac{1}{2\pi i}\log(z_1z_3^*))+ \tau(\frac{1}{2\pi i}\log(z_3z_2^*))
$$
for any tracial state $\tau$, whenever $\|z_1-z_3\|<\delta_2$ and $\|z_2-z_3\|<\delta_2$.
%an additional place to change in LN 5.3
We may further assume that $\dt_2<\min\{\dt_1, \ep/6,1\}.$

Let $\dt>0,$ $\eta>0$
%(in place of $\eta$)
and
$\dt_3>0$ (in place of $\sigma$) required by \ref{LN51} for $\dt_2$ (in place of $\ep$).
%Let $\eta=\min\{d_1, d_2\}.$

Now assume that $u$ and $v$ are two unitaries which satisfy the assumption of the lemma with above $\dt$ and $\eta.$

Since $uv^*\in CU(B),$ $R_{u,v}\in \overline{\rho_B(K_0(B))}.$  Since $\pi(u)=\pi(v)=1,$ we may assume
that $R_{u,v}(t_0)=0$ (see \ref{Ruv}).
It follows that there is $a\in \aff(T(C))$ with $\|a\|<\dt_3/2$ such that
$a-R_{u,v}|_{T(C)}\in \rho_C(K_0(C)).$ Then the image of $ a_\p-R_{u_\p,v_\p}$
is in $\rho_{C\otimes M_\p}(K_0(C\otimes M_\p))),$
 where $a_\p$ is the image of $a$ under the map
 induced by $\imath_\p.$ The same holds for $\q$. Note that
 \beq\label{MT1-22}
 \mu_{\tau\circ \phi_\fr}(I_a)\ge \Delta(a)\tforal \tau\in T(C)
 \eneq
 and for all $a>0$ (and certainly holds for all $a\ge \eta$).
  By Lemma \ref{LN51} there {exist} unitaries $z_\p\in B_\p$ and $z_\q\in B_\q$ such that
 $$[z_\fr]=0\,\,\,{\rm in}\,\,\, K_1(B_\fr),\,\,\,\fr=\p, \q,\,\,\,\|u_\p-z_\p^* v_\p z_\p\|<\delta_2
{\andeqn}
\|u_\q-z_\q^* v_\q z_\q\|<\delta_2.
$$
\vspace{-0.2in}
Moreover,
\beq\label{Mp-19}
\tau(\frac{1}{2\pi i}\log(u_\p^* z_\p^* v_\p z_\p))
=a_\p(\tau)\tforal \tau\in \mathrm{T}(C_\p)
\andeqn \\
\tau(\frac{1}{2\pi i}\log(u_\q^* z_\q^* v_\q z_\q))
=a_\q(\tau)\tforal \tau\in\mathrm{T}(C_\q).
\eneq

We then identify $u_\p, u_q$ with $u\otimes 1$ and
$z_\p$ and $z_\q$ with the elements in the unitization of $C\otimes M_\p\otimes M_\q$ which
is also identified with the unitization of $C\otimes Q.$ In the following computation, we also identify
$T(C)$ with $T(C_\p),$ $T(C_\q),$ and
$T(C_\p),$ or $T(C_\q)$ with $T(C\otimes Q)$
by identifying $\tau$ with $\tau\otimes t,$ where $t$ is the unique tracial state on $M_\p,$ or $M_\q,$ or $Q.$
In particular,
\beq\label{Mp-20}
a_\p(\tau\otimes t)&=&\tau(a)\tforal \tau\in T(C)\andeqn\\
a_\q(\tau\otimes t)&=&\tau(a) \tforal \tau\in T(C)
\eneq
We compute that
by the Exel formula (see 3.5 of {\cite{Lninv}} ),
\begin{eqnarray}
(\tau \otimes t)(\mathrm{bott}_1(u\otimes 1, z_\p ^*z_\q))&=&
 (\tau\otimes t)(\frac{1}{2\pi i}\log(z_\p^* z_\q(u^*\otimes 1)z_\q^* z_\p(u\otimes 1)))\\
&=&{(\tau\otimes t)(\frac{1}{2\pi i}\log(z_\q(u^*\otimes 1) z_\q^* z_\p(u\otimes 1)z_\p^*)}\\
&=&{(\tau\otimes t)(\frac{1}{2\pi i}\log(z_\q(u^*\otimes 1) z_\q ^*(v\otimes 1))}\\
&&+(\tau\otimes t)({1\over{2\pi i}}\log((v^*\otimes 1)z_\p(u\otimes 1)
z_\p^*)\\
&=& (\tau\otimes t)(\frac{1}{2\pi i}\log(u_\q^* z_\q^* v_\q z_\q)\\
&&+(\tau\otimes t)({1\over{2\pi i}}\log(v_\p^* z_\p u_\p z_\p^*))\\
&=&\tau(a) -\tau(a)=0
%&=&{-}(e_{{\q}})_{\sharp}\circ (\imath_{{\q}})_{\sharp}\circ \eta([x])(\tau){+}(e_\p)_{\sharp}\circ (\imath_\p)_{\sharp}\circ\eta([x])(\tau)\\
%&=&{-}(\imath_{\infty})_{\sharp}\circ \eta([x])(\tau){+}(\imath_{\infty})_{\sharp}\circ \eta([x])(\tau)=0
\end{eqnarray}
for all $\tau\in T(C).$ It follows that
\beq\label{Mp-21}
\tau(\mathrm{bott}_1(u\otimes 1, z_\p^* z_\q))=0.
\eneq
for all $\tau\in T(C\otimes Q).$

Since the UHF-algebras $D:=Q$, $M_\p$ or $M_\q$ have unique trace, the map $\rho_C\otimes\mathrm{id}_{K_0(D)}$ is the same as the map $\rho_{C\otimes D}$ if $K_0(C\otimes D)$ is identified as $K_0(C)\otimes K_0(D)$ respectively.

Let $y=\mathrm{bott}_1(u\otimes 1, z_\p^* z_\q)\in \ker \rho_{C\otimes Q}.$

Since $\Q,$ $\Q_\p$ and $\Q_\q$ are flat $\Z$-modules, as in
the proof of 5.3 of \cite{LN-3},
\beq\label{ker}
{\rm ker}\rho_{C\otimes Q}&=&{\rm ker}\rho_C\otimes \Q\\\label{ker-1}
{\rm ker}\rho_{C\otimes M_\fr}&=&{\rm ker} \rho_C\otimes \Q_\fr,\,\,\,\,\fr=\p\andeqn \fr=\q.
\eneq
It follows that there are $x_1, x_2,...,x_l\in \rho_C(K_0(C))$
and $r_1,r_2,...,r_l\in \Q$ such that
$$
y=\sum_{j=1}^l x_j\otimes r_j.
$$
Since $\Q=\Q_\p+\Q_\q,$ one has $r_{j,\p}\in \Q_\p$ and
$r_{j, \q}\in \Q_\q$ such that
$r_j=r_{j,\p}-r_{j,\q}.$
So
$$
y=\sum_{j=1}^l x_j\otimes r_{j, \p}-\sum_{j=1}^l x_j\otimes r_{j, \q}.
$$
Put $y_\p=\sum_{j=1}^l x_j\otimes r_{j, \p}\in {\rm ker}\rho_{C\otimes M_\p}$ and $y_\q=\sum_{j=1}^l x_j\otimes r_{j, \q}\in {\rm ker}\rho_{C\otimes M_\q}.$

It follows from \ref{botext} that there are unitaries $w_\p\in B_\p$ and $w_\q\in B_\q$ such that
\beq\label{Mp-22}
[w_\fr]=0\,\,\,{\rm in}\,\,\, K_1(B_\fr),\,\,\,\fr=\p,\, \q,\,\,\,\|[u_\p\,, w_\p]\|<\dt_1/8,\,\,\, \|[u_\q,\, w_\q]\|<\dt_1/8\andeqn\\
{\rm bott}_1(u_\p, w_\p)=y_\p\andeqn
{\rm bott}_1(u_\q, w_q)=y_\q.
\eneq
Put $W_\p=z_\p^* w_\p \in  {\widetilde{C\otimes M_\p}}$ and $W_\q=z_\q^* w_\q \in {\widetilde{C\otimes M_\q}}.$ Then
\beq\label{Mp-23}
\|u_\p-W_\p^* v_\p W_\p\|<\dt_2+\dt_1/8<\ep/6\andeqn
\|u_\q-W_\q^* v_\q W_\q\|<\dt_2+\dt_1/8<\ep/6.
\eneq
Suppose that $\pi(W_\p)=\lambda_\p$ and $\pi(W_\q)=\lambda_q.$
Replacing $W_\p$ by ${\bar \lambda_\p} W_\p$ and $W_\q$ by ${\bar \lambda_q}W_\q,$ we may assume
that $\pi(W_\p)=\pi(W_\q)=1.$

Note, again,  that $u_\fr=u\otimes 1$ and $v_\fr=v\otimes 1,$  $\fr=\p,\,\q.$
With identification of $W_\fr, w_\fr, z_\fr$ with unitaries in the unitization $C\otimes Q,$
we also have
\beq\label{Mp-24}
\|[ u\otimes 1, \, W_\p^* W_\q]\|<\dt_1/4\andeqn
\eneq
\beq\label{Mp-25}
{\rm bott}_1(u\otimes 1, W_\p^* W_\q) &=&
\mathrm{bott}_1(u\otimes 1,  w^*_\p z_\p z^*_\q  w_\q)\\
&=&
 \mathrm{bott}_1(u\otimes 1, w^*_\p)+
 \mathrm{bott}_1(u\otimes 1, z^*_\p z_\q)+
 \mathrm{bott}_1(u\otimes 1, w_\q)\\\label{Mp-25+}
 &=& -y_\p +(y_\p-y_\q)+y_\q=0.
 \eneq

Let $Z_0=W_\p^* W_\q.$  Then $Z_0\in U_0(C\otimes Q)$ since $C\otimes Q$ has stable rank one.  Note $\pi(Z_0)=\pi(W_\p^*)\pi(W_\q)=1.$ Then it follows from  the choice of $\dt_1,$ (\ref{MT1-22}), (\ref{Mp-25+}) and  \ref{homt} that there is a continuous path of unitaries $\{Z(t): t\in [0,1]\}\subset  {\widetilde{C\otimes Q}}$ such that
$Z(0)=Z_0$ and $Z(1)=1$ and
\beq\label{Mp-26}
\|[u\otimes 1,\, Z(t)]\|<\ep/6\tforal t\in [0,1].
\eneq
Moreover,  by \ref{homt}, we may assume that
\beq\label{Mp-26+1}
\pi(Z(t))=1\tforal t\in [0,1].
\eneq
Define $U(t)=W_\q Z(t).$ Then
\beq\label{Mp-26+2}
U(0)=W_\p,\,\,\, U(1)=W_\q\andeqn \pi(U(t))=1\tforal t\in [0,1].
\eneq
 So, in particular, $U(0)\in B_\p$ and $U(1)\in B_\q.$
Therefore,  by (\ref{Mp-26+2}), $U\in {\widetilde{ C\otimes {\cal Z}_{\p, \q}}}\subset {\widetilde{C\otimes {\cal Z}}}$ is a unitary and, by  (\ref{Mp-23}) and  (\ref{Mp-26}),
\beq\label{Mp-27}
\|u\otimes 1-U^*(v\otimes 1)U\|<\ep/3.
\eneq

\end{proof}

%Note that we assume that $A\otimes {\cal Z}\cong A.$ Let $\imath: A\to A\otimes {\cal Z}$ be the embedding defined by $\imath(a)=a\otimes 1$ for all $a\in A$ and $j: A\otimes {\cal Z}\to A$ such that $j\circ \imath$ is approximately inner.
%Let $V\in A$ be a unitary such that
%\beq\label{Appu-28}
%\|c-V^*j\circ \imath(c)V\|<\ep/3\tforal c\in \{u, v\}.
%\eneq
%Then,  let $w=Vj(U)V^*\in U(A).$
%\beq\label{Appu-29}
%\hspace{-0.8in}\|u- w^*uw\| &\le & \|u-V^*j(u\otimes 1)V\|
%+\|V^*j(u\otimes 1)V- V^*j(U)^*j(v\otimes 1)j(U)V\|\\
%&& +\|V^*j(U)^*j(v\otimes 1)j(U)V-V^*j(U)^*VvV^*j(U)V\|\\
%&<& \ep/3+\|u\otimes 1-U^*(v\otimes 1)U\|+\|j\circ \imath(v)-V^*vV\|\\
%&<& \ep/3+\ep/3+\ep/3=\ep.
%\eneq

The following is a non-unital version of Theorem 4.6 of \cite{Lnexp}.

\begin{thm}\label{Texpl}
Let $A$ be a unital separable simple ${\cal Z}$-stable \CA\,  in ${\cal A}_0.$
Let $C$ be a non-unital hereditary \SCA\ of $A \otimes \mathcal{K}$
with continuous scale and let $B={\tilde C}.$
Suppose that $u\in CU(B).$ Then, for any $\ep>0,$ there exists a self-adjoint element $h\in B$ with $\|h\|\le 1$ such that
\beq\label{Texpl-0}
\|u-\exp(i2\pi h)\|<\ep.
\eneq
%In particular, ${\rm cel}_D(A)\le 2\pi.$
\end{thm}

\begin{proof}
Note that the assumption that $u\in CU(B)$ implies that $\pi(u)=1,$ where
$\pi: B\to \C$ is the quotient map.
Since $C$ has continuous scale, without loss of generality, we may assume that $C\subset M_k(A)$ for some
$k\ge 1.$ To simplify notation, we may further assume, without loss of generality, that
$C\subset A.$
We assume that ${\rm sp}(u)=\T,$ otherwise $u=\exp(ig(u))$ for some (real-valued) continuous  branch of logarithm  with
$\|g(u)\|\le 2\pi.$ Let $\ep>0.$ Let $\phi: C(\T)\to A$ be defined by $\phi(f)=f(u).$ It is a unital monomorphism.
It follows from \ref{Delta}  that there is a non-decreasing function
$\Delta_1: (0,1)\to (0,1)$ such that
\beq\label{Texpl-1}
\mu_{\tau \circ \phi}(I_a)\ge \Delta_1 (a)\tforal \tau\in T(C)
\eneq
for all  arcs  $I_a$ of $\T$ with length $a\in (0,1).$
Define $\Delta(a)=(1/3)\Delta_1(3a/4)$ for all $a\in (0,1).$

Choose $\dt_1>0$ satisfying the following: If $h_1, h_2$ are two selfadjoint elements in any unital \CA\, with $\|h_j\|\le 3\pi,$
$j=1,2,$ such that
$$
\|h_1-h_2\|<\dt_1,
$$
then
$$
\|\exp(ih_1)-\exp(ih_2)\|<\ep/4.
$$
We may assume that $\dt_1<\ep/4.$
Note, by \cite{LN-Range}, for any supernatural number $\p$ of infinite type,
$TR(A\otimes M_\p)=0.$
Let $\p$ and $\q$ be two relatively prime supernatural numbers of
infinite type. Consider $u\otimes 1.$ Denote by $u_\p$ for
$u\otimes 1$ in $A\otimes M_\p.$
For any $\dt_1/2>\ep_0>0,$ by  \ref{expR0},  there is a
selfadjoint element $h_\p\in {\widetilde {C\otimes M_\p}}$ with
${\rm sp}(h_\p)=[-2\pi, 2\pi]$ such that
\beq\label{Texpl-2}
\|u_\p-\exp(ih_\p)\|<\ep_0 \andeqn \tau(h_\p)=0\tforal \tau\in T(C\otimes M_\p).
\eneq

Moreover, for some $1>r>0,$
\beq\label{Texpl-2+}
\lim_{n\to\infty}\tau(|h_\p|^{1/n})<1-r\tforal \tau\in T(C).
\eneq
Let $\Gamma: C([-2\pi, 2\pi])_{s.a.}\to \aff(T(C))$ be defined
by
\beq\label{Texpl-3}
\Gamma(f)(\tau)=(\tau\otimes t)(f(h_\p))\tforal f\in C([-2\pi, 2\pi])_{s.a.}
\eneq
and for all $\tau\in T(C),$
where
$t$ is the unique tracial state on $M_\p.$

Let  (note we now assume that $C\subset A$)
\beq\label{Texpl-3+}
d(\tau)=\sup\{\tau(c): 0\le c\le 1\andeqn c\in C\}\tforal \tau\in T(A).
\eneq
Note that $\inf_{\tau\in T(A)}(d(\tau))>0$ since $A$ is simple.
Since $C$ has continuous trace, $d\in Aff(T(A)).$ Define
$\Gamma_1: C([-2\pi, 2\pi])_{s.a.}\to \Aff(T(A))$ by
\beq\label{Texpl-3+1}
\Gamma_1(f)(\tau)=d(\tau) \Gamma(f)(\tau/d(\tau))\tforal \tau\in T(A).
\eneq

Let $\eta>0,$ $\dt>0$ and let ${\cal G}$ be a finite subset as required by \ref{Mapp} for $\dt_1/4$ (in place of $\ep$) and $\Delta.$
Note $\Delta(a)=(1/3)\Delta_1(3a/4)$ for all $a\in (0,1).$
Choose $\ep_0$ sufficiently small, so the following holds (see Lemma 3.4 of \cite{Lninv}):
For any unitary $v\in C\otimes M_\p,$ if
$\|u_\p-v\|<\ep_0,$ then
\beq\label{Texpl-4}
\mu_{\tau\circ \psi}(I_a)\ge \Delta(a)\tforal \tau\in T(C)
\eneq
and for all arcs $I_a$ with length $a\ge \eta,$
where $\psi: C(\T)\to A\otimes M_\p$ is the \hm\, defined by
$\psi(g)=g(v)$ for all $g\in C(\T),$ and
\beq\label{Texpl-5}
|\tau(g(u_\p))-\tau(g(v))|<\dt\tforal \tau\in T(C)
\eneq
and for all $g\in {\cal G}.$
Note each $\tau\in C\otimes M_\p$ may be written as
$s\otimes t,$ where $s\in T(C)$ is any tracial state and
$t\in T(M_\p)$ is the unique tracial state.
It follows from  3.8 of \cite{Lnexp} that there exists a selfadjoint
element $h_0\in A$  with ${\rm sp}(h)=[-2\pi, 2\pi]$ such that
\beq\label{Texpl-6}
\tau(f(h_0))=\Gamma_1(f)(\tau)=({\tau\otimes t\over{d(\tau)}})(f(h_\p))d(\tau)\tforal f\in C([-2\pi, 2\pi])
\eneq
and for all $\tau\in T(A).$  It follows that
$$
d_\tau(|h_0|)< d(\tau)(1-r).
$$
for all $\tau\in  T(A)$ (see (\ref{Texpl-2+})).
Let $\{e_n\}$ be an
approximate identity  for $C$ ($e_n$ may not be projections).
Since $C$ has continuous scale, we may assume
$$
d_\tau(|h_0|)<\tau(e_n)\tforal \tau\in T(A)
$$
and for some $n\ge 1.$ Note that $0\le e_n\le 1.$
So, in particular,
$$
d_{\tau}(|h_0|)<d_{\tau}(e_n)\tforal \tau\in T(C).
$$
By strict comparison,  we may assume that $h_0\in e_nCe_n.$
Note that
\beq\label{Texpl-7}
\tau(h_0)=0\tforal \tau\in T(C).
\eneq
Define $v_1=\exp(i h_0)\in A$ and denote by
$\psi_1: C(\T)\to A$ the \hm\, given by
$\psi_1(f)=f(v_1)$ for all $f\in C(\T).$
Note that, by (\ref{Texpl-6}),
\beq\label{Texpl-8}
\tau(g(v_1))=(\tau\otimes t)g(\exp(ih_\p))\tforal \tau\in T(C)
\eneq
and for all $f\in C(\T).$ By the choice of $\ep_0,$ as in (\ref{Texpl-4}) and (\ref{Texpl-5}),
\beq\label{Texpl-9}
\mu_{\tau\circ \psi_1}(I_a)=
\mu_{(\tau\circ \otimes t)\circ \psi_1}(I_a)\ge \Delta(a)\tforal \tau\in T(C)
\eneq
and for all arcs $I_a$ with length $a\ge \eta.$  Moreover,
\beq\label{Texpl-10}
|\tau(g(u))-\tau(g(v_1))|
=|(\tau\otimes t)(g(u_\p))-(\tau\otimes t)(g(\exp(ih_\p)))|<\dt
\eneq
for all $\tau\in T(A)$ and for all $g\in {\cal G}.$
Note since $\tau(h_0)=0$ for all $\tau\in T(C),$ $v_1\in CU(C).$
It follows from \ref{Mapp} that there exists a unitary $W\in {\widetilde C\otimes {\cal Z}}$ such that
\beq\label{Texpl-11}
\|(u\otimes 1)-W^*(v_1\otimes 1)W\|<\dt_1/4.
\eneq

Denote by $\imath: A\to A\otimes {\cal Z}$ defined by $\imath(a)=a\otimes 1$ and define
$j: A\otimes {\cal Z}\to A$ such that $j\circ \imath$ is approximately inner. Since
$A$ is separable, there is $e\in C_+$ with $0\le e\le 1$ which is strictly positive in $C.$
Let $f_c(t)=1$ if $t\in [c, 1],$ $f_c(t)=0$ if $t\in [0,c/2]$ and linear in $[c/2, c].$
So $\{f_{1/n}(e)\}$ and $\{f_{1/n}(e)\otimes 1\}$ form  approximate identities for $C$ and
for $C\otimes {\cal Z},$ respectively.
Choose $1>c>0$
such that
\beq\label{Texpl-12}
\|(f_c(e)\otimes 1)W^*(v_1\otimes 1)W-W^*(v_1\otimes 1)W\|
&<&\dt_1/16\pi \andeqn\\\label{Texpl-12+1}
  \|W^*(v_1\otimes 1)W(f_c(e)\otimes 1)-W^*(v_1\otimes 1)W\|&<&\dt_1/16\pi
\eneq
Let $z\in U(A)$ such that
\beq\label{Texpl-13}
\|z^*j\circ \imath(a)z-a\|<\dt_1/16\pi
\eneq
for $a\in \{u, v_1, h, f_c(e), e, f_{c/2}(e)\}.$
Then,
\beq\label{Texpl-14}
\|u- z^*j(W^*(v_1\otimes 1)W)z\|<\dt_1/16\pi.
\eneq
Moreover,
\beq\label{Texpl-14+1}
\|f_{c/2}(e)z^*j(f_{c}(e)\otimes 1)z-z^*j(f_{c}(e)\otimes 1)z\|<\dt_1/8\pi
\eneq

Put $H=z^*j(W^*)j(h_0\otimes 1)j(W)z, $  $H_1=z^*j(f_c(e)\otimes 1)j(W^*)j(h_0\otimes 1)j(W)j(f_c(e)\otimes 1)z$ and $h=f_{c/2}(e)z^*j(W^*)j(h_0\otimes 1)j(W)zf_{c/2}(e).$
Then $\|H\|\le 2\pi, \|H_1\|\le 2\pi$ and $\|h\| \le 2\pi.$
We note that $z^*j(W^*)j(h_0\otimes 1)j(W)z\in A$ and $C$ is hereditary,
$h\in C$ (recall that we assumed that $C\subset A$ at the beginning of this proof).
Since $v_1=\exp(ih_0),$ by (\ref{Texpl-14}),
\beq\label{Texpl-14+2}
\|u-\exp(iH)\|<\dt_1/16\pi.
\eneq
It follows from  (\ref{Texpl-12}) and (\ref{Texpl-12+1}) that
\beq\label{Texpl-15}
\|H-H_1\|<4\dt_1/16=\dt_1/4
\eneq
and by (\ref{Texpl-14+1}),
\beq\label{Texpl-16}
\|H_1-h\|<4\dt_1/8=\dt_1/2.
\eneq
Therefore
\beq\label{Texpl-17}
\|H-h\|<\dt_1.
\eneq
By the choice of $\dt_1,$  we have
\beq\label{Texpl-16}
\|\exp(i H)-\exp(i h)\|<\ep/4.
\eneq
It follows from (\ref{Texpl-14+2}) and (\ref{Texpl-16}) that
\beq\label{Texpl-17}
\|u-\exp(ih)\|<\ep.
\eneq
The point is that now  $h\in C.$
\end{proof}

\section{Unitaries in a small corner}

The main goal of this section is to prove a technical lemma of \ref{cornerunit}
which is 
used to compute $K_1$ groups of $J$ (as in \ref{MinIdeal}). 

Key steps in the proof of Lemma \ref{cornerunit} are 
Lemmas \ref{lem:Pre2Diagonals} and \ref{lem:BasisStep},  
which are attempts to generalize  
a well-known argument of 
Elliott (\cite{Elliott1974}, proof of Theorem 2.4).      
Elliott's argument roughly  says the following:  If $B$ is a nonunital
UHF-algebra and $u$ is a unitary in $M(B)$, then $u$ can be approximately
factored into the product of two appropriate block diagonal unitaries. Each
block diagonal unitary will turn out to be homotopic to a block diagonal unitary
with a trivial block (i.e., a block which is a projection).

The argument of Elliott has been generalized by a number of authors
(e.g., \cite{Mingo} Part 2, \cite{LinWVNII} Lemma 2.6 and 
\cite{ZhK_1Flow} Lemma 1.6).  However, unlike the case \cite{Mingo}, $J$ is very different from the full multiplier algebra $M(B\otimes {\cal K}).$ Moreover, $B$ does not have real rank zero which were  cases in \cite{Elliott1974},  \cite{LinWVNII}
and in \cite{ZhK_1Flow}. In fact, our main case is $B={\cal Z}$ which completely lacks of projections.  Nevertheless
we will try to carry out Elliott's idea without projections in hand.

Let $A$ be a C*-algebra. Recall  that $A\otimes {\cal K}$ has \emph{has
strict comparison of positive elements by traces} 
if for all $a, b \in (A \otimes \K)_+$,
the condition that
$d_{\tau}(a) < d_{\tau}(b)$ whenever $d_{\tau}(b) < \infty$ for every
nonzero densely defined lower semicontinuous trace $\tau$ on $A \otimes \K$,
implies that $a \preceq b$. 

Often, we will abbreviate the above by just saying that $A$ has \emph{strict 
comparison of
positive elements} or $A$ has \emph{strict comparison}.\\

%\green{ Suggestion: Call the above "$A\otimes {\cal K}$ has the generalized strict comparison for positive elements".}

\begin{prop}\label{Pstrictcomp}
Let $A$ be a unital simple \CA\, which has the strict comparison (in the sense of \ref{cuntz}).
Suppose that $A$ has only finitely many extremal tracial states. 
Then $A\otimes {\cal K}$ has the strict comparison in the sense of above.
\end{prop}

\begin{proof}
Suppose that $a, \, b\in A\otimes {\cal K}_+$ such that 
$d_\tau(a)<d_\tau(b)$ for all tracial states $\tau\in T(A)$ for which $d_\tau(b)<\infty.$ 
Without loss of generality, we may assume that $0\le a,\, b\le 1.$
Let $\{\tau_1, \tau_2,...,\tau_k\}$ be the set of extremal tracial states of $A.$ 
We may assume that 
$d_{\tau_i}(b)<\infty$  for all $i\le k_1$ and $d_{\tau_i}(b)=\infty$ if $k\ge i>k_1.$ 
Set
\beq\label{Pstr-1}
\dt=\min\{d_{\tau_i}(b)-d_{\tau_i}(a): 1\le i\le k_1\}.
\eneq
Let $\{e_{i,j}\}$ be a system of matrix unit for ${\cal K}.$ Set $e_n=\sum_{i=1}^n e_{i,i},$ $n=1,2,....$ 

For any $1/2>\ep>0,$ there exists an integer $n_1\ge 1$ such that
\beq\label{Pstr-2}
\|e_{n_1}ae_{n_1}-a\|<\ep/4.
\eneq
Put $a_1=e_{n_1}ae_{n_1}.$ Then $d_\tau(a_1)<\infty$ for all $\tau\in T(A).$
Moreover, since $a_1\lesssim a,$ 
\beq\label{Pstr-2}
d_\tau(a_1)<d_\tau(a) \rforal \tau\in T(A).
\eneq
There exists an integer $m_0\ge 1$ such that
\beq\label{Pstr-3}
\tau_i(b^{1/m_0})>d_{\tau_i}(b)-\dt/2\rforal 1\le i\le k_1\andeqn\\
\tau_i(b^{1/m_0})>d_{\tau_i}(a_1) +\dt/2\rforal  k\ge i>k_1.
\eneq
Choose an integer $n_2\ge n_1$ such that
\beq\label{Pstr-4}
\|e_{n_2}b^{1/m_0}e_{n_2}-b^{1/m_0}\|<\min\{\ep/2, \dt/4\}.
\eneq
Put $b_1=e_{n_2}b^{1/m_0}e_{n_2}.$ 
Then we compute that
\beq\label{Pstr-5}
d_{\tau}(b_1)\ge \tau(b_1)>d_\tau(a_1)\rforal \tau\in T(A).
\eneq
By viewing both $b_1$ and $a_1$ as elements in $M_{n_2}(A),$ we obtain an element $x\in e_{n_2}(A\otimes {\cal K})e_{n_2}$ such that
\beq\label{Pstr-6}
\|x^*e_{n_2}b^{1/m_0}e_{n_2}x-a_1\|<\ep/4
\eneq
One then finds an element $y\in A\otimes{\cal K}_+$ such that 
\beq\label{Pstr-7}
\|x^*e_{n_1}y bye_{n_1}x-a_1\|<\ep/2.
\eneq
It follows that
\beq\label{Pstr-8}
\|x^*e_{n_2}ybye_{n_2}x-a\|<\ep.
\eneq
Since this holds for every $\ep>0,$ we conclude that $a\lesssim b.$
\end{proof}

\begin{prop} \label{prop:TracialDetermination}
 Suppose that $A$ is a simple, separable, stably finite C*-algebra such 
that $A \otimes \K$ has strict comparison of positive elements by 
traces.
Suppose,  in addition, that $A$ has stable rank one.

   Suppose that $p, \, q\in M(A\otimes {\cal K})\setminus A\otimes {\cal K}$ 
are two projections such that $\tau(p)=\tau(q)$ for every 
densely defined,
lower semicontinuous trace $\tau$.
  
Then $p \sim q$, i.e.,
there exists a partial isometry $v\in M(A\otimes {\cal K})$ such that
$$q = v^*v \makebox{  and  } p = v v^*.$$
   \end{prop}

\begin{proof}
Set $C=A\otimes{\cal K}.$ 
Let $a\in pCp$ and $b\in qCq$ be two strictly positive elements (in the
respective C*-algebras).
Then neither are projections.
By the assumption,
$$
d_\tau(a)=d_\tau(b)
$$
for every densely defined, lower semicontinuous trace $\tau$.

Since $C$ has strict comparison for positive elements and $A$ has stable rank one, one obtains
a partial isometry $v\in C^{**}$ such that 
$vv^* = p$, $v^* v = q$, 
$xv\in C$ for all $x\in {\overline{Ca}}$ and $vy\in C$ for all $y\in {\overline{bC}}.$
Moreover,
$$
v^*\overline{aCa}v={\overline{bCb}}.
$$
(See, for example, \cite{LinCuntzSemigroup} Theorem 1.7.)
For any $c\in C,$ since $q\in M(C),$ $qc\in C.$ It follows that $qc\in \overline{bC}.$ Therefore
$$
vc=vqc\in C.
$$
Similarly, $cv=cpv\in C.$ Therefore $v\in M(C).$
\end{proof}

The next result is standard (see, for example, 
\cite{JensenThomsen} Lemma 1.3.2).  We sketch the
argument for the convenience of the reader: 

\begin{lem} Suppose that $B$ is a stable C*-algebra.
Then 
$$H_B \cong B$$
where $H_B$ is the standard Hilbert $B$-module, and the isomorphism is
(unitary) isomorphism of Hilbert $B$-modules.

\label{lem:StableHilbertModule} 
\end{lem}

\begin{proof}[Sketch of proof:]  
Since $B$ is a stable C*-algebra, let $\{ W_n \}_{n=1}^{\infty}$ 
be a sequence of elements of $M(B)$ such that the following hold:
\begin{enumerate}
\item $W_n^* W_n = 1$ (i.e., $W_n$ is an isometry), for all $n \geq 1$. 
\item $W_m^* W_n = 0$ for all $m \neq n$.
\item $1_{M(B)} = \sum_{n=1}^{\infty} W_n W_n^*$, where the sum converges
strictly in $M(B)$. 
\end{enumerate}

Let $\Phi : H_B \rightarrow B$ be given by 
$$\Phi(b_1, b_2, b_3, ....) =_{df} \sum_{n=1}^{\infty} W_n b_n$$
for all $(b_1, b_2, b_3, ...) \in H_B$.

One can check that $\Phi$ is an (unitary) isomorphism of Hilbert $B$-modules.
\end{proof}
 
We now fix some notation that will be used in the next proposition and
its corollary.
For a simple C*-algebra $A$, we let $Ped(A)$ denote the Pedersen ideal
of $A$ (see \cite{PedersenIdeal}).  Recall that $Ped(A)$ is hereditarily
closed, and $Ped(A)$ can be characterized as the minimal dense two-sided
(algebraic) ideal of $A$ (\cite{PedersenIdeal}, \cite{LaursenSinclair}). 
For a nonzero positive element $e \in Ped(A)$,
let $T_e(A)$ denote the set of all densely defined, lower semicontinuous
traces $\tau$ on $A$ such that $\tau(e) = 1$.
When endowed with the topology of pointwise convergence on $Ped(A)$,
$T_e(A)$ is a Choquet simplex.  (See \cite{TikuisisToms} 
Proposition 3.4.  Note that their notation for $T_e(A)$ is different
from ours, and their terminology is also slightly different.) 

Of course, if $A$ is unital then $T_1(A) = T(A)$, where $T(A)$ is as
defined in Definition \ref{Dtrace}.

\begin{prop}\label{L43}
Suppose that $A$ is a simple, separable, stably finite C*-algebra, 
and let $e$ be a nonzero positive element of $Ped(A)$.

Suppose, in addition, that  
for every bounded strictly positive affine lower semicontinuous 
function $f: T_e(A)\to (0, \infty),$  there exists a non-zero 
$a\in (A\otimes {\cal K})_+$ which is not Cuntz equivalent to  a projection
such that $d_\tau(a)=f(\tau)$ for all $\tau\in T_e(A).$
 %$\mathcal{Z}$-stable.

Then for every bounded, strictly positive, affine, lower semicontinuous
function $f : T_e(A) \rightarrow (0, \infty)$,
there exists a projection
$p \in M(A \otimes \K) \setminus (A \otimes  \K)$
such that
$$\tau(p) = f(\tau)\tforal \tau\in T_e(A).$$
%for every $\tau \in T(A)$
\label{prop:RangeOfTrace}
\end{prop}

\begin{proof}

Let $a \in C_+$ ($ = (A \otimes \K)_+$) be such that
$a$ is not
Cuntz equivalent to
a projection in $C$ and
\begin{equation}
d_{\tau}(a) = f(\tau)
\label{HBee}
\end{equation}
for all $\tau \in T_e(A)$.
By the Kasparov Absorption Theorem (\cite{Kasparov} Theorem 2)
and by Lemma \ref{lem:StableHilbertModule},
let $p \in M(C) \setminus C 
\cong L_C(C)\setminus K_C(C)$   
be a projection 
such that
$$p C \cong \overline{a C}$$
where the isomorphism is (unitary) isomorphism of Hilbert $C$-modules.
(In the above, $L_C(C)$ is the C*-algebra of all adjointable 
operators on the Hilbert module $C$, and $K_C(C)$ is the C*-subalgebra
of all compact adjointable
operators. See, for example, \cite{WeggeOlsen} Chapter 15.) 

Since $C$ is separable, we can find a strictly positive element 
$b \in p C p$.
Hence,
$$\overline{b C} \cong \overline{a C}.$$

Hence, by \cite{ORT} 
Proposition 4.3 (see also \cite{LinCuntzSemigroup}), 
$$d_{\tau}(b) = d_{\tau}(a)$$
for all $\tau \in T_e(A)$. 

But we have that for all $\tau \in T_e(A)$,
$f(\tau) = d_{\tau}(a)$ and $\tau(p) = d_{\tau}(b)$.
Hence, for all $\tau \in T_e(A)$,
$$f(\tau) = \tau(p)$$
as required.
\end{proof}

\begin{rem}
We note that by \cite{BrownPereraToms} Theorem 5.5,
if $A$ is a unital simple exact finite and $\mathcal{Z}$-stable
C*-algebra and if we take $e =_{df} 1$  
then $A$ satisfies the hypotheses of Proposition \ref{L43}.
\end{rem}

We can generalize Proposition \ref{L43} to the case where
the lower semicontinuous function is unbounded or takes the
value $\infty$.

\begin{cor}\label{C45}
Suppose that $A$ is a simple, separable, stably finite C*-algebra, and
let $e$ be a nonzero positive element of $Ped(A)$.  

Suppose, in addition, that $A$ has the property that, for every
bounded strictly positive affine lower semicontinuous function
$f : T_e(A) \rightarrow (0, \infty)$, there exists a nonzero
$a \in (A \otimes \K)_+$ which is not Cuntz equivalent to a projection
such that $d_{\tau}(a) = f(\tau)$ for all
$\tau \in T_e(A)$.\\
(E.g., $A$ can be unital simple exact finite and $\mathcal{Z}$-stable and 
$e = 1_{A}$.)

  Then for every strictly positive, affine, lower semicontinuous
function $f : T_e(A) \rightarrow (0, \infty]$, there exists
a projection $p \in M(A \otimes \mathcal{K}) \setminus
(A \otimes \mathcal{K})$
such that
$$\tau(p) = f(\tau)\tforal \tau\in T_e(A).$$
\label{Extra}
\end{cor}

\begin{proof}
   Let $\{ q_n \}_{n=1}^{\infty}$ be a sequence of pairwise orthogonal
projections in $M(A \otimes \mathcal{K})$ such that
$q_n \sim 1$ for all $n \geq 1$ and
$\sum_{n=1}^{\infty} q_n = 1$,
where the sum converges in the strict topology on $M(A \otimes \mathcal{K})$.

   Next, $f$ is the pointwise limit of a strictly increasing sequence
of strictly positive affine continuous functions on $T_e(A)$ (see
\cite{ABP} Lemma 5.3;  see also \cite{Goodearl} Theorem 11.12 or
\cite{Edwards} (Edwards' Separation Theorem); see also \cite{TikuisisToms}
Proposition 4.1 and Lemma 4.2).
In other words,
let $\{ f_n \}_{n=1}^{\infty}$ be a sequence of strictly positive
affine continuous functions on $T_e(A)$ with
$f_n < f_{n+1}$ for all $n$ such that $f_n \rightarrow f$ pointwise
(where $\infty$ is allowed to be the limit at a point) .

   For all $n \geq 1$,
let $g_n$ be the strictly positive affine continuous function on $T_e(A)$
that is given by
$g_n =_{df} f_n - f_{n-1}$  (where $f_0 =_{df} 0$).
For each $n \geq 1$, apply Proposition \ref{L43}
to get a projection $p_n \in q_n M(A \otimes \mathcal{K}) q_n$
such that
$$\tau(p_n) = g_n(\tau)\tforal \tau\in T_e(A).$$

Let $p =_{df} \sum_{n=1}^{\infty} p_n$.
Then the sum converges strictly and $p$ is a projection
in $M(A \otimes \mathcal{K}) - (A \otimes \mathcal{K)}$ such that
$$\tau(p) = f(\tau)\tforal \tau\in T_e(A).$$

\end{proof}

The next statement follows from \cite{KaftalNgZhang}: 
\begin{thm}\label{T46}
Let $A$ be a unital simple stably finite C*-algebra
with strict comparison of positive elements by traces such that 
$T(A)$ has finitely many extreme points.

Suppose that $a, b \in M(A \otimes \K)_+$ with $a \in Ideal(b)$
and $d_{\tau}(a) < d_{\tau}(b)$ for all $\tau \in T(A)$ for which
$d_{\tau}(b) < \infty$.

Then $a \preceq b$.  
\label{StrictComp} 
\end{thm}

\begin{proof}
This follows immediately from \cite{KaftalNgZhang} Theorem 5.4.
\end{proof}

The next lemma is standard (e.g., see Lemma 2.1 in \cite{Elliott1974}).
%{\bf I assume that you have checked Effros's book. I do not remember which
%one is earlier. Effros put several perturbation results in his
%little AMS book.}

\begin{lem}\label{L44}
For every $\epsilon > 0$, there exists
$\delta > 0$ such that
the following holds:

If $C$ is a unital C*-algebra and $p, q \in C$ projections such that
$$p q \approx_{\delta} p$$
then there exists a unitary $u \in C$
such that
\beq u p u^* \leq q\andeqn
\| u - 1 \| < \epsilon.\eneq
\label{lem:ElliottPerturbation}
\end{lem}

The next lemma is an observation of R\o rdam (\cite{RorUHF1} section 4).

\begin{lem}
Let $C$ be a unital C*-algebra
and let $x \in C$ be a nilpotent element (i.e., $x^2 = 0$).

Then
$$x \in \overline{GL(C)}$$
where the closure is in the norm topology.
\label{lem:nilpotent}
\end{lem}

%\begin{proof}
%Note that
%$$(\epsilon 1 + x)( \epsilon 1 - x)
%= ( \epsilon 1 - x) (\epsilon 1 + x)
%= \epsilon^2 1 + \epsilon x  - \epsilon x - x^2
%= \epsilon^2 1.$$

%I.e., $\epsilon 1 + x \in GL(C)$.

%Moreover,
%$$\epsilon 1 + x \approx_{\epsilon} x.$$\\\\
%\end{proof}

\noindent\textbf{NOTE:}
For the rest of this section, unless otherwise stated,
$A$ is a unital separable simple $\mathcal{Z}$-stable
C*-algebra with unique
tracial state $\tau$ which is the only normalized quasitrace and
${J}$ is the unique proper C*-ideal of
$M({A} \otimes {\cal K})$ which properly 
contains ${A} \otimes {\cal K}$ (see \ref{MinIdeal} and \ref{Dconscale}).

\begin{lem}\label{L49}
Let $p \in M(A\otimes {\cal K})$ 
%{ J}\setminus {A} \otimes {\cal K}$ 
be a projection, and
suppose that $a \in p {J} p$ is a nonzero positive element such that
$$2 d_{\tau}(a)  < \tau(p)$$
and
there exists positive $c \in p Jp\setminus A \otimes \K$ such that $c\perp a$.

Then for every $\epsilon > 0$,
there exists a projection $q \in p
{ J} p\setminus p ({A} \otimes {\cal K}) p$
such that
$$d_{\tau}(a) < \tau(q) < d_{\tau}(a) + \epsilon\andeqn
q a \approx_{\epsilon} a.$$
\label{lem:SpecialAU}
\end{lem}

\begin{proof}

For simplicity, let us assume that $\| a \| \leq 1$ and $\epsilon < \min\{1, {\tau(p)-2d_\tau(a)\over{2}}\}$.

Let $h : [0, 1] \rightarrow [0,1]$ be the unique continuous function
such that
\[
h(t) =
\begin{cases}
1 & t \in [\epsilon/64, 1] \\
0 & t \in [0, \epsilon/128] \\
\makebox{linear on  } [\epsilon/128, \epsilon/64]
\end{cases}
\]

Hence,   $p - h(a) \in p(A \otimes \mathcal{K})p$
is a positive element such that
$$\tau(p - h(a)) = \tau(p) - \tau(h(a))
\ge  \tau(p) - d_{\tau}(a) > d_{\tau}(a) + \ep.$$
Also, since $c \in Her(p - h(a))$, $p - h(a) \notin A \otimes \K$.

Hence, by  Proposition \ref{prop:RangeOfTrace},
%let $r_1\in p { J} p - p (A \otimes {\cal K})p$ 
let $r_1\in J\setminus A\otimes {\cal K}$  be a projection
such that
$$
d_{\tau}(a) < \tau(r_1) < d_{\tau}(a) + \epsilon <\tau(p-h(a)).
$$
    By Theorem \ref{T46}, 
$$
r_1\lesssim p-h(a).
$$

It follows that there is a projection 
 $r \in Her(p - h(a))$ such that 
 $[r]=[r_1].$ In particular, $\tau(r)=\tau(r_1).$
Hence,
$$r \perp (a - \epsilon/32)_+.$$
By Theorem  \ref{StrictComp} again,
$a \preceq r$. Hence,
let $x \in p {J} p$ be an element such that
$$x^* x = (a - \epsilon/8)_+\andeqn
x x^* \leq r.$$
Let
$$x = v |x|$$
be the polar decomposition of $x$ (in $A^{**}$).
Then
$$
(a - \epsilon/8)_+ = |x|^2\andeqn
v |x|^2 v^* \leq r.
$$

Note that since $(a - \epsilon/8)_+ \perp r$,
$x^2 = 0$.
Hence, by Lemma \ref{lem:nilpotent},
$$x \in \overline{GL(p {J}p)}.$$

Hence, by Theorem 5 in   \cite{PedersenUE},
let $u \in p {J}p$ be a unitary such that

$$u (a - 3 \epsilon/4)_+ u^* = v (a - 3 \epsilon/4)_+v^* \leq r.$$

Hence,

$$(a - 3 \epsilon/4)_+ \leq u^* r u.$$

In particular,
$$u^* r u a \approx_{\epsilon} a.$$

Taking $q =_{df} u^* r u$,
we are done.

\end{proof}

\begin{lem}\label{L47}
Let $p \in {J} \setminus ({A} \otimes {\cal K})$ be a projection  and let
$u \in p {J} p$ be a unitary.

Suppose that $p_1 \in p {J}p \setminus p ({A} \otimes {\cal K}) p$ is a projection
such that
$$100 \tau(p_1) < \tau(p).$$

Then for every $\epsilon > 0$ and for  every projection
$q \in J\setminus A \otimes \K$ with $q \perp p$, there exist
projections $e_1, e_2, e_3\in J$ such that
$$p_1 = e_1 < e_2 < e_3 < p + q$$
where the inequalities are strict, and
there exists a unitary $w \in (p+q) {J}(p+q)$
such that
\beq\nonumber
\| w - (u + q) \| < \epsilon\andeqn
w e_1 w^* \leq e_2 \leq w e_3 w^* < p + q.
\eneq
Moreover,
\beq\nonumber
e_2\preceq p+q-e_2.
\eneq
\label{lem:Pre2Diagonals}
\end{lem}

\begin{proof}
By the virtue of Propositions \ref{prop:TracialDetermination} 
and\ref{L43}, replacing $q$ by a subprojection if necessary, we may assume that
$$\tau(q) < (\tau(p) - \tau(p_1))/10^{10000}.$$

By Theorem \ref{StrictComp},  Proposition \ref{prop:RangeOfTrace}
and Proposition \ref{prop:TracialDetermination},
we can decompose $q$ into a direct sum of non-zero projections
$$q = q' + q'' + q'''$$
such that $q', q'', q''' \in J\setminus (A \otimes \K)$.

Let $\dt_5>0$ (in place of $\dt$) associated with $\ep/100$ (in place of $\ep$) 
be given by Lemma \ref{lem:ElliottPerturbation}. Let $\dt_{j-1}>0$ (in place of $\dt$) associated with $\dt_j/100$ (in place of $\ep$) given by
Lemma \ref{lem:ElliottPerturbation}, $j=5,4,3,2,1.$
We may assume that $\delta_5 < \epsilon/100$, and that 
for all $j$, $\dt_{j-1} < \dt_j/100$. 

%By Lemma \ref{lem:ElliottPerturbation},
%let
%$$0 < \delta_1 < 100 \delta_1 <  \delta_2 < 100 \delta_2 <  \delta_3 <
%100 \delta_3 < \delta_4 < 100 \delta_4 <  \delta_5 = \epsilon/100$$
%be such that
%if we plug $\delta_j/100$ into Lemma \ref{lem:ElliottPerturbation},
%then we get $\delta_{j-1}$.

Take $e_1 =_{df} p_1$.

By hypothesis,
$$\tau(e_1) < \tau(p)/100.$$
Hence,
\beq\nonumber
e_1+ue_1u^*\in pJp\andeqn d_{\tau}(e_1 + u e_1 u^*) < \tau(p)/50.
\eneq
%$$e_1 + u e_1 u^* \leq p$$
%and
%$$d_{\tau}(e_1 + u e_1 u^*) < \tau(p)/50.$$

%{\bf (I could not see why $e_1 + u e_1 u^* \leq p$.)}

It follows from
Lemma \ref{lem:SpecialAU} that there exists a projection
$q_1 \in (p + q') {J} (p + q')$
such that
\beq\label{L47-n10}
d_{\tau}(e_1 + u e_1 u^*) < \tau(q_1) < d_{\tau}(e_1 + u e_1 u^*) +
\min\{ \delta_1, \tau(p) \}/10^{100},
\eneq
$$q_1 e_1 \approx_{\delta_1} e_1\andeqn
q_1 u e_1 u^* \approx_{\delta_1} u e_1 u^*.$$

By Lemma \ref{lem:ElliottPerturbation} and the definition
of $\delta_1$, we can find a projection
$e_2 \in (p+q') {J} (p+q')$ such that
$$\| e_2 - q_1 \| < \delta_2/50$$
(and hence $e_2 \sim q_1$)
and
$$e_1 \leq e_2.$$

As a consequence, (since $\delta_1 < \delta_2/100$),
$$e_2 (u e_1 u^*) \approx_{\delta_2} u e_1 u^*.$$

It follows from Lemma \ref{lem:ElliottPerturbation}
and the definition of $\delta_2$ that there is a unitary
$w_1 \in (p+q') {J} (p+q')$ such that
$$\| w_1 - (u+q') \| < \delta_3/50$$
and
$$w_1 e_1 w_1^* \leq e_2.$$

Since
$$d_{\tau}(e_2 + w_1^* e_2 w_1)
\leq 2 ( \tau(p)/50 + \min\{ \delta_1, \tau(p)
\}/10^{100} ) \leq \tau(p)/25 + 2 \tau(p)/10^{100},$$
and by Lemma \ref{lem:SpecialAU},
let
$q_2 \in (p+q'+q'') {J} (p+q'+q'')$ be a projection such that
$$d_{\tau}(e_2 + w_1^* e_2 w_1) < \tau(q_2)
< d_{\tau}(e_2 + w_1^* e_2 w_1) + \min \{ \delta_2, \tau(p) \}/10^{100},$$
$$q_2 e_2 \approx_{\delta_2/100} e_2$$
and
$$q_2 w_1^* e_2 w_1 \approx_{\delta_2/100} w_1^* e_2 w_1.$$

By Lemma \ref{lem:ElliottPerturbation} and the definition
of $\delta_2$,
there exists a projection $e_3 \in (p+q'+q'') {J} (p+ q' + q'')$ such that
$$\| e_3 - q_2 \| < \delta_3/50$$
(and hence $e_3 \sim q_2$), and
$$e_2 \leq e_3.$$
Hence,
$$e_3 (w_1^* e_2 w_1) \approx_{\delta_3} w_1^* e_2 w_1.$$
Hence,
$$(w_1+q'') e_3 (w_1+q'')^* e_2 \approx_{\delta_3} e_2.$$

Note that since
$$e_1 \leq e_2 \leq e_3,$$

$$(w_1+q'') e_1 (w_1+q'')^* \leq (w_1+q'') e_3 (w_1+q'')^*.$$
Also, by our construction of $e_2$,
$$(w_1+q'') e_1 (w_1+q'')^* \leq e_2.$$
Hence,
$$(w_1+q'') (e_3 - e_1) (w_1+q'')^* (e_2 - (w_1+q'') e_1 (w_1+q'')^*)
\approx_{\delta_3} (e_2 - (w_1+q'') e_1 (w_1+q'')^*).$$

Hence, by Lemma \ref{lem:ElliottPerturbation} and by
the definition of $\delta_3$,
there exists a unitary
$w'_2 \in (p+q'+q'' - w_1 e_1 w_1^*) {J} (p+q'+q'' - w_1 e_1 w_1^*)$
with
$$\| w'_2 - (p+q'+q'' - w_1 e_1 w_1^*) \| < \delta_4/100$$
and
$$e_2 - w_1 e_1 w_1^* \leq w'_2(w_1+q'') (e_3 - e_1) (w_1+q'')^* w'^*_2.$$

Hence, if we take
$$w =_{df} (w'_2 + w_1 e_1 w_1^*) (w_1+q'') + q'''$$
then
$w \in (p+q) {J} (p+q)$ is a unitary,
\beq\nonumber
&&\| w - (u+q) \| < \epsilon,\,\,\,
p_1 = e_1 < e_2 < e_3 < p,\\
&&\andeqn w e_1 w^* < e_2 < w e_3 w^* < p
\eneq
as required.
Finally, since $q_1\sim e_2,$ by (\ref{L47-n10}) and by applying
Theorem \ref{StrictComp}, we also have
$$
e_2\preceq p+q-e_2.
$$
\end{proof}

\begin{lem} Let $p \in {J} \setminus({A} \otimes {\cal K})$ be a projection
and $u \in p {J} p$ be a unitary.
Let $1/2^{10} > r > 0$.
Then for every projection $q \in J \setminus A \otimes \K$ with
$q \perp p$,
there is a (norm-) continuous path of unitaries
$\{ u(t) : t \in [0,1] \}$ in $U((p+q) {J} (p+q))$ such that
$$
u(0) = u +q, \makebox{  } u(1) = ((p+q-f_2) + u_1)((p+q- f_1) + v_1)
$$
where $v_1 \in U(f_1 {J} f_1)$,
$u_1 \in U(f_2 {J} f_2)$,
$\tau(p +q - f_j) > r \tau(p)$ for %all $\tau \in T({A})$,
and $f_j\in (p+q)J(p+q)\setminus  A \otimes \K$ is a projection, $j = 1,2$.
\label{lem:BasisStep}
\end{lem}

\begin{proof}

By Proposition \ref{prop:RangeOfTrace} and Theorem \ref{StrictComp},
there is a projection  $p_1 \in p {J} p \setminus p({A} \otimes {\cal K})p$
such that
$\tau(p_1) > r \tau(p)$ and $100 \tau(p_1) < \tau(p)$
(equivalently,
$\beta < \tau(p_1) < \tau(p)/100$ for some $\beta < \tau(p)/2^{10}$.)

Since $q \in J \setminus (A \otimes \K)$ is a projection, by Propositions
\ref{prop:RangeOfTrace} and \ref{prop:TracialDetermination},
we can decompose
$q$ into orthogonal projections
$$q = q' \oplus q''$$
where $q', q'' \in J\setminus (A \otimes \K)$.

It follows from  Lemma \ref{lem:Pre2Diagonals} that
there exists a unitary $w_0 \in (p+q'){J}(p+q')$  with
$$\| w_0 - (u+q') \| < \epsilon/10000$$
(which implies that $w_0$ is path-connected to $u + q'$ in
$(p+q') {J}(p+q')$)
and
there are projections  $e_1, e_2, e_3 \in (p+q'){J} (p+q')\setminus ({A} \otimes {\cal K})$
such that
$$p_1 = e_1 < e_2 < e_3 <  p+q' < p+q$$
(where the inequalities are strict),
and
$$w_0 e_1 w_0^* < e_2 < w_0 e_3 w_0^* < p+q' < p+q.$$
Moreover,
\beq\label{L48-n1}
e_2\preceq p+q-e_2.
\eneq
It follows that
\beq\label{L48-n2}
e_1\preceq p+q-e_1.
\eneq
%For simplicity, let us also denote ``$w + q''$" by ``$w$".
Define $w=w_0+q''.$
Hence, we have that
\beq\label{L48-n0}
\| w - (u + q) \| < \epsilon/10000,\\
p_1 = e_1 < e_2 < e_3 < e_4 =_{df} p+q,\\
w e_1 w^* < e_2 < w e_3 w^* < p+q,
\eneq
and there exists a positive $c \in (p+q) J (p+q) - (A \otimes \K)$ such that
$$c \perp e_3 + w e_3 w^*.$$

From the above and by Proposition \ref{prop:TracialDetermination},
\beq\nonumber
&&e_1 \sim w e_1 w^*,\,\,\,
e_2 - e_1 \sim e_2 - we_1 w^*, \\\nonumber
&&e_3 - e_2 \sim  w e_3 w^* - e_2 \andeqn
p+q - e_3 \sim p+q - w e_3 w^*.
\eneq
Recall that $e_4 =_{df} p+q$.
Put $e_0 =_{df} 0$.
Hence, for $j= 1, 2$,
there is a unitary  $w_{2j} \in (e_{2j} - e_{2j - 2}) {J} (e_{2j} - e_{2j - 2})$
such that
$$w_{2j} (e_{2j-1} - e_{2j-2}) w_{2j}^* =
w e_{2j-1} w^* - e_{2j-2}.$$

Define
$$z =_{df} w_2 \oplus w_4\andeqn y =_{df} z^* w,$$
two  unitaries  in
$(p+q) {J} (p+q).$
%Let $y \in p {J} p$ be the unitary that is given by
%$$y =_{df} z^* w.$$ a unitary in $pJp.$
%{\bf (I do not see why $y$ can be made into a
%unitary $pJp.$ I assume it is a typo.)}
Then
$w = z y.$
Clearly, by definition,
\beq\nonumber
(e_{2j} - e_{2j-2}) z = w_{2j} =  z (e_{2j} - e_{2j-2})\andeqn\\\nonumber
\tau(e_2 - e_0) = \tau(e_2) \geq \tau(e_1) > r \tau(p).
\eneq
Also, since
$$z e_1 z^* = w e_1 w^*,$$
$$e_1 = z^* w e_1 w^* z.$$
In other words,
\beq\nonumber
e_1 = y e_1 y^*\,\,\,{\rm or}\,\,\,
e_1 y = y e_1.
\eneq
Hence, $$y = y_1 \oplus y_3$$
where $y_1 \in e_1 {J} e_1$ is a unitary and
$y_3 \in (p+q - e_1) {J} (p+q - e_1)$ is a unitary.
Moreover,
\beq\label{L48-n3}
\tau(e_2)>\tau(e_1) =\tau(p_1) > r \tau(p).
\eneq
 By (\ref{L48-n1}) and (\ref{L48-n2}),
 there are partial isometries $Z_1, Z_2\in (p+q)J(p+q)$ such that
 $$
 Z_i^*Z_i=e_i,\,\,\, Z_iZ_i^*=_{df}e_i'\le p+q-e_i,\,\,\, i=1,2.
 $$
 Define a unitary $U_i\in (p+q)J(p+q)$ by
 $$
 U_i=Z_i+Z_i^*+(p+q - e_i -e_i'),\,\,\,i=1,2.
 $$
 We may write
 \beq\label{L48-n4}
 w=zy=(w_2\oplus U_2w_2^*U_2^*U_2w_2U_2^*w_4)(y_1\oplus U_1y_1^*U_1^*U_1y_1U_1^*y_3).
 \eneq
(In the above, we are simplifying notation by using ``$U_2 w_2 U_2^*$" 
to denote 
$U_2 w_2 U_2^*  + (p + q -e_2 - e'_2)$, a unitary in 
$(p+q - e_2) J (p+q - e_2)$. Similarly, we are
using ``$U_1 y_1 U_1^*$" to denote $U_1 y_1 U_1^* + (p + q - e_1 - e'_1)$.)

Put $f_i=e_i,$ $i=1,2.$  So $f_i\in (p+q)J(p+q)\setminus A\otimes {\cal K}$ are projections.
$$u_1=U_2w_2U_2^*w_4\in (p+q-f_2)J(p+q-f_2)\andeqn v_1=U_1y_1U_1^*y_3\in (p+q-f_1)J(p+q-f_1).$$
Then, by (\ref{L48-n4}) and (\ref{L48-n0}),  $u+q$ is connected to
$$
(e_2+u_1)(e_1+v_1)
$$
via a norm continuous path in $U(p+q)J(p+q)).$
Moreover, we compute that, by (\ref{L48-n3}), (\ref{L48-n1}) and
(\ref{L48-n2}),
$$
\tau(p+q-f_i) \geq \tau(e_i)>r\tau(p),\,\,\, i=1,2.
$$
%Consider
%$$w = zy = (w_2 \oplus w_4)(y_1 \oplus y_3)
%= (w_2 \oplus w_2^* w_2 w_4) ( y_1 \oplus y_1^* y_1 y_3).$$
%(Here, we are using that $e_2 \preceq p+q - e_2$,
%and (under this subequivalence),
%representing $w_2$ as a unitary in $p+q - e_2$ by adding the trivial projection
%on the complement.  Similar for $e_1 \preceq p+q - e_1$ and $y_1$.)

%As a consequence,
%$w$ is norm path connected to
%$$( e_2 \oplus w_2 w_4)( e_1 \oplus y_1 y_3).$$
%Since $u + q$ is norm path connected to $w$, we are done.
\end{proof}

\begin{lem}
 Let $p \in {J} \setminus ({A} \otimes {\cal K})$ be a projection
and $u \in p {J} p$ be a unitary.
Let $1 > r > 0$.
Then for every projection $q \in J\setminus A\otimes \K$ with $q \perp p$,
there is an integer $m\ge 1$ and a (norm-) continuous path of unitaries
$\{ u(t) : t \in [0,1] \}$ in $U((p+q) {J}(p+q))$ such that
$$u(0) = u + q, \makebox{  } u(1) = ((p+q-f_1) + u_1)((p+q- f_2) + u_2) ...
((p+q - f_m) +u_m) $$

where $u_j \in U(f_j {J} f_j)$,
$\tau(p+q - f_j) > r \tau(p)$, and
$f_j \in (p+q)J(p+q)\setminus  A \otimes \mathcal{K}$ is a projection,  
%all $\tau \in T(A)$,
$j = 1,2, ..., m$.
\label{lem:smallprojections}
\end{lem}

\begin{proof}
There is an integer $n\ge 1$ such that
$1-(1-s)^n>r,$ where $s=(3/4)2^{-10}.$
By Theorem \ref{StrictComp}, Proposition \ref{prop:TracialDetermination}
and Proposition \ref{prop:RangeOfTrace},
there are mutually orthogonal projections 
$q_l\in qJq\setminus A\otimes {\cal K},$ $l=1,2,...,n$, 
such that $q=\sum_{l=1}^n q_l.$
%let $\{ q_l \}_{l=1}^{\infty}$ be a sequence of pairwise orthogonal
%projections in $J - (A \otimes \K)$ such that
%$$q = \sum_{l=1}^{\infty} q_l$$
%where the sum converges in the strict topology on $M(A \otimes \K)$.
%Let $s =_{df} 1/2^{10}$.

It suffices to prove the following statement:\\ 

There are projections
\beq\label{L49-n1}
&&f_1, f_2,...,f_{2^k}\in (p+\sum_{j=1}^kq_j)J(p+\sum_{j=1}^kq_j)\setminus A\otimes {\cal K},\\\label{L49L-n2}
&&{\rm unitaries}\,\,\, u_i\in U(f_jJf_j)\,\,\,{\rm such\,\,\, that}\,\,\, \tau(f_j)<\tau(\sum_{j=1}^kq_j)+(1-s)^k\tau(p),
\eneq
$j=1,2,..,k,$
and $u$ is connected to $\prod_{j=1}^{2^k}((p+q-f_j)+u_j)$  by a norm continuous path of unitaries in $(p+q)J(p+q)$ for
all $1\le k\le n.$

%that for all $N \geq 1$,
%we can find $f_j, u_j, m$ satisfying the statement of
%the lemma, but with the statement that
%$$\tau(p+q - f_j) > r \tau(p)$$
%(i.e., $(1 - r) \tau(p) + \tau(q) > \tau(f_j)$)
%for all $j$, replaced with (the stronger statement that)
%
%$$(1-s)^k \tau(p) + \sum_{l=1}^k \tau(q_l) > \tau(f_j)$$
%for all $j$ and for all $k \geq 1$.
We prove this  by induction on $k$.

The case $k = 1$
follows immediately from Lemma \ref{lem:BasisStep} with $q_1$ in place of $q$
and $s$ in place of $r.$
%(where we plug $q_1$ into the $q$ of Lemma \ref{lem:BasisStep}).\\
%Induction step.

Suppose that we have proven the statement for $k$.
We now prove the statement for $k + 1$.
%By the induction hypothesis, find $f_j$ and $u_j$ so that the statement
%holds for $k$.

%We have that for $1 \leq j \leq m$,
%$$(1 - s)^k \tau(p) + \sum_{l=1}^k \tau(q_l) > \tau(f_j).$$

Now for $1 \leq j \leq k$,
apply Lemma \ref{lem:BasisStep} to $u_j$ (in place of $u$), $f_j$ (in place of $p$),  $s$ (in place of $r$) and $q_{k+1}$ (in place of $q$) to get
projections $f_{j,l}\in (f_j+q_{k+1})J(f_j+q_{k+1})\setminus A\otimes {\cal K}$ ($l=1,2$) and unitaries  $u_{j,l}\in f_{j,l}Jf_{j,l}$
such that
$u_j + q_{k+1}$ is connected to 
$(f_j + q_{k+1}- f_{j,1}+u_{j,1})(f_j+q_{k+1} - f_{j,2} +u_{j,2})$
by a norm continuous path of unitaries in $(f_j+q_{k+1})J(f_j+q_{k+1})$ and
\beq\label{L49-n10}
\tau(f_j+q_{k+1}-f_{j,l})>s\tau(f_j).
\eneq
From this and the induction hypothesis, it follows that
\beq\label{L49-n11}
\tau(f_{j,l}) &<&\tau(f_j)+\tau(q_{k+1})-s\tau(f_j)\\
&=&\tau(q_{k+1})+(1-s)\tau(f_j)\\
&<&\sum_{j=1}^{k+1}\tau(q_j)+(1-s)^{k+1}\tau(p).
\eneq
The statement above holds for $k+1$  by renaming the
$f_{j,l}s$ and $u_{j,l}s$
which completes the induction.
%Hence,
%$$\tau(f_{j,l}) < (1 - s) \tau(f_j) + \tau(q_{k+1})
% < (1-s) ((1 - s)^k \tau(p) + \sum_{l=1}^k \tau(q_l)) + \tau(q_{k+1})
%< (1 - s)^{k+1} \tau(p) + \sum_{l=1}^{k+1} \tau(q_l)$$
%as required.
%The induction is complete.
\end{proof}

\begin{lem}
Let $p, q, r, s \in {J}\setminus {A} \otimes {\cal K}$ be  projections such that
\beq\nonumber
q, r, s \leq p\\
\{ q,r \} \perp s \\
\andeqn
\tau(q) \leq \tau(s) \leq \tau(r).
\eneq
Suppose that $u \in q J q$ is a unitary.
Then there is a continuous path of unitaries $\{u(t): t\in [0,1]\}\subset pJp$ such that
$$
u(0)=u+(p-q)\andeqn u(1)=v+(p-r),
$$
where
%the unitary
%$u + (p - q)$ (in $p{J}p$)
%is norm-path-connected (in $p {J} p$) to a unitary
%$v + (p - r)$
where $v \in r J r$ is a unitary.
\label{lem:movingprojections}
\end{lem}

\begin{proof}

Since $s \perp q$ and $q \preceq s$,
we may view $u + (p - q)$ as
$u + (p - q) = diag(u, s, p - q - s)
= diag( u, u^* u  , p - q - s )$
(with coordinates in $(q, s, p - q - s)$)
which is norm path connected to
$diag(q, u, p - q - s)$, which has the
form $w + (p - s)$ for some unitary
$w \in s J s$.

But since $s \perp r$ and $s \preceq r$,
$w + (p-s)$
can be viewed
as
$w + (p-s) = diag(w, r, p - s - r) =
diag(w, w^* w, p - s - r)$
(with coordinates in $(s, r, p - s - r)$)
which is norm path connected to
$diag(s, w, p - s - r)$, which has the
form $v + (p - r)$
for some unitary
$v \in rJr$.
\end{proof}

\begin{lem}\label{cornerunit}
Let $p \in {J}\setminus {A} \otimes {\cal K}$ be a projection
and let $q \in {J}\setminus {A} \otimes {\cal K}$ be another projection
with $p \perp q$
and $q \preceq p$.

Let $u \in p {J} p$ be a unitary.
Then there exists a continuous path of unitaries $\{u(t): t\in [0,1]\}\subset U((p+q)J(p+q))$ such that
\beq\label{L411-n1}
u(0)=u+q\andeqn u(1)=p+v,
\eneq
where $v\in U(qJq).$
%unitary
%$v \in q {J} q$ such that
%$u + q$ is norm-path connected to $p + v$
%in $(p + q)J(p+q)$.\\
\end{lem}

\begin{proof}
The result of the lemma follows immediately from
Lemma \ref{lem:smallprojections} and
Lemma \ref{lem:movingprojections}.\\
By Proposition \ref{prop:RangeOfTrace},
Proposition \ref{prop:TracialDetermination} and
Theorem \ref{StrictComp},
we can decompose $q$ into pairwise orthogonal projections
$$q = q' \oplus q'' \oplus q'''$$
where $q', q'', q''' \in J \setminus (A \otimes \K)$
and $\tau(q') < \tau(q'') < \tau(q''')$, which implies that
$q' \preceq q'' \preceq q'''$.

By applying  Lemma \ref{lem:smallprojections} and by passing to a
unitary which connects with $u+q'$ by a continuous path of unitaries
in $U((p+q)J(p+q)),$ without loss of generality, we may write
\beq\label{L411-n2}
u+q'=\prod_{j=1}^m ((p+q'-f_j)+u_j)
\eneq
such that $u_j\in U(f_jJf_j),$ where $f_j\le p+q'$  and
$f_j\preceq q'',$  $j=1,2,...,m.$
Note that
$$
f_j, q'', q'''\le p+q,\,\,\,f_j\perp q'', f_j\perp q''', q'''\perp q'' \andeqn f_j\preceq q''\preceq q''',\,\,\,j=1,2,...,m.
$$
Thus, by Lemma \ref{lem:movingprojections}, for each $j,$
there is a continuous path of unitaries
$\{w_j(t): t\in [0,1]\}\subset (p+q)J(p+q)$ such that
\beq\label{L411-n3}
w_j(0)=(p+q'+q''+q'''-f_j)+u_j\andeqn w_j(1)=v_j+(p+q'+q''),
\eneq
where $v_j\in U(q'''Jq'''),$ $j=1,2,...,m.$
Define
$$
u(t)=\prod_{j=1}^m w_j(t)\tforal t\in [0,1].
$$
Then $\{u(t):t\in [0,1]\}$ is a continuous path of unitaries in
$(p+q)J(p+q).$ Moreover,
$$
u(0)=u+q\andeqn u(1)=v+p,
$$
where $v=\prod_{j=1}^m (v_j+q'+q'')\in U(qJq).$

%(In slightly more detail:  By Proposition \ref{prop:RangeOfTrace},
%Proposition \ref{prop:TracialDetermination} and
%Lemma \ref{lem:StrictComp},
%we can decompose $q$ into pairwise orthogonal projections
%$$q = q' \oplus q'' \oplus q'''$$
%where $q', q'', q''' \in J - (A \otimes \K)$
%and $\tau(q') < \tau(q'') < \tau(q''')$, which implies that
%$q' \preceq q'' \preceq q'''$.
%Now apply Lemma \ref{lem:smallprojections} to $u$, $p$ and $q'$ to
%get $f_j$, $u_j$, $m$ where
%^$\tau(f_j) < \tau(q'')$ ($1 \leq j \leq m$).
%For $1 \leq j \leq m$, apply Lemma \ref{lem:movingprojections}
%to $f_j, u_j, q'', q'''$ to move the nontrivial parts (of the unitaries)
%to $q'''$.)
\end{proof}

\section{The main results}

\begin{lem}\label{sequenceproj}
Let $A$ be a unital \CA\, and $B= A\otimes {\cal K}.$
Then $M(B)$ contains a sequence of mutually orthogonal projections $\{p_n\}$ each of which
is equivalent to the identity $1_{M(B)},$  $p_n\notin B,$
$n=1,2,...,$  and
$\sum_{n=1}^{\infty} p_n$ converges strictly to $1_{M(B)}.$
\end{lem}

\begin{proof}
This is known. One can have a partition of $\N$ into a sequence
$\{I_n\}$ of mutually disjoint infinite subsets  such that $1\in I_n$ and
$I_n\cap \{1,2,...,n-1\}=\emptyset.$
Let $\{e_{i,j}\}$ be a system of matrix unit for  ${\cal K}.$  We identify
$e_{ii}$ with $1_A\otimes e_{ii}.$
Define
$$
p_n=\sum_{j\in I_n} e_{jj}.
$$
The convergence is in the strict topology. Let $e_n=\sum_{i=1}^n e_{ii}.$
Then $\{e_n\}$ forms an approximate identity for $B.$
Note that, by the construction,
$$
p_me_n=0\tforal m>n.
$$
It follows that
$\sum_{k=1}^n p_k$ converges strictly to $1_{M(B)}.$
\end{proof}

The following holds in a much more general situation.

\begin{lem}\label{seqsmallpro}
Let $A$ be a unital separable simple
${\cal Z}$-stable C*-algebra with unique tracial state
$\tau$, and let
$B={A}\otimes {\cal K}.$
Let $J$ be the smallest ideal of $M(B)$ which properly contains $B$.

Then, for any sequence of positive numbers
$\{\alpha_n\}$ with $\alpha_n\ge \alpha_{n+1}$
and $\sum_{n=1}^{\infty}\alpha_n<\infty,$ there exists a sequence
of mutually orthogonal projections $\{q_n\}\in M(B)-B$ such that
$q_{n+1}\lesssim q_n,$ $\tau(q_n)=\alpha_n,$ $n=1,2,...,$ such that
$\sum_{n=1}^{\infty}q_n$ converges in the strict topology and $p=\sum_{n=1}^{\infty}q_n\in J.$
\end{lem}

\begin{proof}
Let $\{p_n\}$ be as in \ref{sequenceproj}.

By Proposition \ref{prop:RangeOfTrace} and since $p_n \sim 1_{M(B)}$,
let $q_n \in J-B$ be such that $q_n \leq p_n$ and
$\tau(q_n) = \alpha_n$ for all $n \geq 1$.
Then $\sum_n q_n$ converges strictly to an element of $J$.
\end{proof}

\begin{cor}\label{Cssp}
Let $A$ be a unital separable simple $\mathcal{Z}$-stable C*-algebra with
unique tracial state $\tau$, let
$B={A}\otimes {\cal K}$, let $J$ be the smallest ideal of $M(B)$ which properly
contains $B$, and let $p\in M(B)$ be a projection such that
$\tau(p)=\infty.$
Then, for any sequence of positive numbers
$\{\alpha_n\}$ with $\alpha_n\ge \alpha_{n+1}$ and
$\sum_{n=1}^{\infty}\alpha_n<\infty,$ there exists a sequence
of mutually orthogonal projections $\{q_n\}\in M(B) -B $ such that
$q_{n+1}\lesssim q_n,$ $\tau(q_n)=\alpha_n,$ $n=1,2,...,$ such that
$\sum_{n=1}^{\infty}q_n$ converges in the strict topology to a projection $q\in J$ such that
$q\le p.$
\end{cor}

\begin{proof}
This follows from \ref{seqsmallpro} and the fact that $p\sim 1_{M(B)}.$
\end{proof}

\begin{lem}\label{unitaryconn}
Let $B$ be a  non-unital and $\sigma$-unital simple \CA,
$J\subset M(B)$ be an ideal
containing $B$ such that $J/B$ is purely infinite and simple and let
$p\in J$ be a non-zero projection.
Suppose that $u\in pM(B)p$ is a unitary such that $[u]=0$ in $K_1(pM(B)p).$
Then, for any $\ep>0,$ there are two selfadjoint elements
$H_1, H_2, \in pM(B)p$  and a unitary $w\in p+pBp$ such that
$\|H_1\|\le \pi,$  $\|H_2\|<\ep$  and
\beq\label{unitaryconn-0}
u=w\exp(iH_1)\exp(iH_2).
\eneq
In particular, $w^*u\in U_0(pM(B)p).$
\end{lem}

\begin{proof}
Note that $pM(B)p/pBp$ is purely infinite and simple.
Let $\pi: pM(B)p\to pM(B)p/pBp$ be the quotient map.
Suppose that $u$ is as in the lemma. Then
$\pi(u)\in  U_0(pJp/pBp).$
Note that $pJp/pBp$ is a unital purely infinite  simple \CA.
It follows from a result of Chris Phillips  (\cite{Ph}) that, for any $\ep>0,$
there  are selfadjoint elements $h_1, h_2\in pM(B)p/pBp$ such that
\beq\label{unitaryconn-1}
\pi(u)=\exp(i h_1)\exp(i h_2)\andeqn \|h_1\|\le \pi\andeqn \|h_2\|<\ep.
\eneq
There are  selfadjoint elements $H_1, H_2\in pM(B)p$
such that  $\|H_1\|\le \pi,$  $\|H_2\|\le \ep$  and $\pi(H_j)=h_j,$ $j=1,2.$
Let $u_j=\exp(i H_j),$ $j=1,2.$
Then
\beq\label{unitaryconn-2}
w=uu_2^*u_1^*\in p+pBp.
\eneq
It follows that $u=w\exp(iH_1)\exp(H_2)$ and $w^*u\in U_0(pM(B)p).$
\end{proof}

\begin{lem}\label{GroupL}
Let $A\in {\cal A}_0$ be a unital separable simple ${\cal Z}$-stable \CA\ with unique
tracial state $\tau$,
let $B=A\otimes {\cal K}$ and let $J\subset M(B)$ be an ideal containing $B$
such that $J/B$ is purely infinite and simple.
Suppose that $p\in J$
% such that
%$\hat{p}$ is continuous on $T(A)$
and $u\in pM(B)p\setminus B$ be a unitary. Then, for any projection $q\in J\setminus B$
%with
%$\hat{q}$ continuous on $T(A)$
and  $pq=0=qp,$ any $\ep>0,$ there is a unitary $v\in qM(B)q$ and selfadjoint elements $h_1, h_2\in (p+q)M(B)(p+q)$
and $h_3\in \C(p+q)+ (p+q)B(p+q)$ such that
\beq\label{GroupL-1}
&&\|h_1\|<\ep/2,\,\,\,\|h_2\|\le \pi,\,\,\, \|h_3\|\le 2\pi,\\
&&\|(u+v)-(p+q)\exp(i h_1)\exp(ih_2)\exp(ih_3)\|<\ep\andeqn\\
&&(u+v)-(p+q)\exp(i h_1)\exp(ih_2)\exp(ih_3)\in (p+q)B(p+q).
\eneq
\end{lem}

\begin{proof}
There is, by \ref{C45}  and \ref{T46}, a projection $q_1\in J\setminus B$ such that $q_1\le q$
and there is a unitary $z\in (p+q)M(B)(p+q)$ such that
$$
z^*q_1z\le p.
$$
It follows from Lemma \ref{cornerunit}  that there is a unitary $v_1\in q_1Jq_1$ such
that
\beq\label{GroupL-2}
u+v_1+(q-q_1)\in U_0((p+q)J(p+q)).
\eneq
Put $P=p+q.$
It follows from \ref{unitaryconn} that there are selfadjoint elements
$h_1, h_2\in PJP$ with $\|h_1\|<\ep/2$ and $\|h_2\|\le \pi$ such that
$$
u+v_1+(q-q_1)=\exp(ih_1)\exp(ih_2)w,
$$
where $w\in P+PBP.$
Since $B$ is simple and has stable rank one,
there exists $w_1\in q+qBq$ such that $[w_1]=-[w].$ Thus
$w(w_1+p)\in U_0(P+PBP).$
Let
$$
w(w_1+p)=\prod_{k=1}^m \exp(ia_k),
$$
where $a_k\in (\C P+PBP)_{s.a.}.$ Let $\pi(a_k)=\af_k,$ $k=1,2,...,m.$
Since $w(w_1+p)\in P+PBP,$
$\sum_{k=1}^m\af_k=2N\pi$ for some integer $N.$
%Let $\lambda=\sum_{k=1}^m \af_k
By replacing $a_k$ by
$a_k-\af_kP,$ we may assume that $\pi(a_k)=0,$ $k=1,2,...,m.$
In other words, $a_k\in PBP,$ $k=1,2,...,m.$
Choose $b\in (qBq)_{s.a.}$ such that
\beq\label{GroupL-3}
\tau(b)=\sum_{k=1}^m \tau(a_k).
\eneq
Set
\beq\label{GroupL-4}
v_2=q \exp(-ib).
\eneq

Now let $v=(v_1+(q-q_1))w_1 v_2$
Then
$$
u+v=\exp(ih_1)\exp(ih_2)w(w_1+p)(v_2+p).
$$
Note that
$$
w (w_1+p)(v_2 + p)=P\prod_{k=1}^m \exp(ia_k) \exp(-ib).
$$
But
$$
\tau(-b)+\sum_{k=1}^m \tau(a_k)=0.
$$
It follows that
$$
w(w_1+p)(v_2 + p)\in CU(\C P+PBP).
$$
By \ref{Texpl}, there is $h_3\in PBP_{s.a.}$ with $\|h_3 \|\le 2\pi$ such that
$$
\|w (w_1+p)(v_2 + p)-P\exp(ih_3)\|<\ep.
$$
It follows that
$$
\|(u+v)- P\exp(ih_1)\exp(ih_2)\exp(ih_3)\|<\ep.
$$
\end{proof}

\begin{lem}\label{Jappp} Let $A$ be a unital separable simple $\mathcal{Z}$-stable C*-algebra
with unique tracial state $\tau$ and  let
$J$ be the smallest ideal in $M(A \otimes \K)$ that properly contains
$A \otimes \K$. 
Then, for any finite subset ${\cal F}\subset J,$  any projection $e\in J$ and any $\ep>0,$ there exists a
projection $p\in J$ such that
\beq\label{Jup-1}
e\le p\andeqn \|xp-x\|<\ep \tforal x\in {\cal F}.
\eneq
%Then $J$ has an approximate identity consisting of projections.
\end{lem}

\begin{proof}
Let $1/2>\ep>0,$ finite subset ${\cal F}\subset J$ and a projection $e\in J$ be given.
We may assume, without loss of generality, that $\|x\|\le 1$ for all $x\in {\cal F}.$ 
Let $\{e_\af\}$ be an approximate identity for $J.$ There exists $\af$ such that
\beq\label{Jappp-2}
\|xe_\af -x\|<\ep/64\andeqn \|ee_\af-e\|<\ep/64.
\eneq

%First we consider the case  that zero is not an isolated point of $sp(e_\af).$ 
We may assume that $\|e_\af\|>1/2.$ 
Put $b=(e_\af-\ep/64)_+.$  
Let
\[
f_{\ep/65, \ep/64}(t) =_{df}
\begin{cases}
1 & t \in [\ep/64, 1] \\
0 & t \in [0, \ep/65] \\
\makebox{linear on }[ \ep/65, \ep/64].  
\end{cases}
\]
 
Since $e_{\af} \in J$, $1_{M(A \otimes \K)} - f_{\ep/65, \ep/64}(e_\af)$
is a full element of $M(A \otimes \K)$, and 
$d_{\tau}(1_{M(A \otimes \K)} - f_{\ep/65, \ep/64}(e_\af)) = \infty$.    
Hence, by Theorem \ref{StrictComp}, 
there is a positive element $c\notin A \otimes \K,$ $c\perp b.$ 
Recall also, by \ref{Dconscale}, that $d_{\tau}(b) < \infty$.  
It follows from Lemma \ref{lem:SpecialAU} (with $p=1_{M(A\otimes {\cal K})}$) that there is a projection $q\in J$  such that
\beq\label{Jappp-3}
\|bq-b\|<\ep/64.
\eneq
Therefore
\beq\label{Jappp-4}
\|e_\af q-e_\af\|<\ep/64+\|bq-b\|+\ep/64<3\ep/64.
\eneq
We estimate that
\beq\label{Jappp-5}
\|xq-x\|<\|xe_\af q-xq\|+\|xe_\af q-xe_\af\|+\|xe_\af-x\|\\
<\ep/64+3\ep/64+\ep/64=5\ep/64\andeqn\\
\|eq-e\|<5\ep/64.
\eneq
There exists a unitary $u\in M(A\otimes {\cal K})$ such that 
$\|u-1\|<10\ep/64$ such that 
$u^*qu\ge e.$ Let $p=u^*qu.$

%In the case that $0$ is an isolated point, there is a continuous function 
%$f\in C_0((0,1])_+$ with $0\le f\le 1$ such that $q=f(e_\af)$ is a projection 
%and $qe_\af=e_\af.$ 

\end{proof}

\begin{thm}\label{K1}
Let $A\in {\cal A}_0$ be a unital separable simple $\mathcal{Z}$-stable C*-algebra
with unique tracial state $\tau$  and let
$J$ be the smallest ideal in $M(A \otimes \K)$ that properly contains
$A \otimes \K$.

Then
$K_1({J})=0.$ In fact, for any $u\in {\tilde{J}},$
$$
{\rm cel}(u)\le 7\pi.
$$
\end{thm}

\begin{proof}
Let $u\in M_n({\tilde {J}}).$  Let $\Pi: M_n({\tilde { J}})\to M_n$
be the quotient map.
Let $z\in M_n$ such that $\Pi(u)=z.$ Identify $z$ with the  scalar matrix in $M_n({\tilde {J}})$ and put
$v=uz^*.$ Then $v\in {\widetilde M_n({J})}.$
Moreover, $\Pi(v)=1.$ Since $M_n({J})\cong {J},$ without
loss of generality, we may assume that
$u\in 1+ {J}$  but we need to show that ${\rm cel}(u)\le 6\pi.$

Let $\ep>0.$   By applying \ref{Jappp}, 
%Since ${J}$ has an approximate identity consisting of projections, without loss of generality,
we may assume that there is a projection $P\in {J}$ such that
$$
u=(1-P)+u_0,
$$
where $u_0\in P{J}P$ is a unitary.

%By reconsidering an  approximate identity of ${J}$ consisting of projections, one find a non-zero
%projection $e_0$ in $(1-P){J}(1-P).$ 
Note that $\tau(1-P)=\infty.$
Then, by applying \ref{Cssp},
%\ref{seqsmallpro},
one obtains a sequence of mutually orthogonal nonzero
projections $\{q_n\}$ in $J - A \otimes \K$ such that
$\sum_{n=1}^{\infty} q_n$ converges in the strict topology and $\sum_{n=1}^{\infty}q_n\in (1-P){J}(1-P).$
By \ref{GroupL},  putting $P=q_0,$ there are  a sequence of unitaries, $u_n\in q_n{J}q_n$ and
 sequences of selfadjoint elements
$\{h_j^{2n-1}\}\subset (q_{2n-1}+q_{2n}){J}(q_{2n-1}+q_{2n})$ and
 $\{h_j^{(2n)}\}\subset (q_{2n}+q_{2n+1}){J}(q_{2n}+q_{2n+1})$ such that
 \beq\label{k1-1}
&& \|h_1^{(n)}\|\le \pi,\,\,\,\|h_2^{(n)}\|\le 2\pi,\\
 &&\|(u_{2n-1}+u_{2n})-\exp(ih_1^{(n)})\exp(ih_2^{(n)})\|<\ep/2^n\andeqn\\
 &&\|(u_{2n}+u_{2n+1})-\exp(ih_1^{2n})\exp(i h_2^{(2n)})\|<\ep/2^n,
 \eneq
$n=0, 1,2,....$
Since $\sum_{n=1}^{\infty} q_n$ converges in the strict topology and by (\ref{k1-1}),
so do $\sum_{n=0}^{\infty} u_n,$  $\sum_{n=1}^{\infty} u_n,$ $\sum_{n=0}^{\infty} h_1^{(n)},$
$\sum_{n=1}^{\infty} h_1^{(n)},$ $\sum_{n=0}^{\infty} h_2^{(n)}$ and
$\sum_{n=1}^{\infty}  h_2^{(n)}.$
Define
\beq\label{k1-2}
&&U_0=\sum_{n=0}^{\infty}u_n,\,U_1=\sum_{n=1}^{\infty} u_n\\
&&H_{0,1}=\sum_{n=0}^{\infty} h_1^{(n)},\, H_{0,2}=\sum_{n=0}^{\infty} h_2^{(n)}\\
&&H_{1,1}=\sum_{n=1}^{\infty} h_1^{(n)}\andeqn H_{1,2}=\sum_{n=1}^{\infty}  h_2^{(n)}.
\eneq
Then
\beq\label{k1-3}
&&U_0=u_0+U_1,\\\label{k1-4}
&&\|U_1-\exp(iH_{1,1})\exp(i H_{1,2})\|<\ep\\\label{k1-5}
&&\|U_0-\exp(i H_{0,1})\exp(iH_{02})\|<\ep\\\label{k1-6}
&&\|H_{0,1}\|\le \pi\andeqn \|H_{0,2}\|\le 2\pi.
\eneq
It follows from (\ref{k1-4}) that $[U_1]=0$ in $K_1({\tilde{J}}).$  Then, by (\ref{k1-5}) that
$[u_0]=0$ in $K_1({\tilde {J}}).$ It follows that $[u]=0$ in $K_1({\tilde {J}}).$
Since $u$ is arbitrarily chosen, we conclude that
$$K_1({\tilde {J}})=0.$$
Moreover, by (\ref{k1-4}), (\ref{k1-5}) and (\ref{k1-6}), one computes that
$$
{\rm cel}(u)\le 3\pi +3\pi=6\pi.
$$
So in general,
$$
{\rm cel}(u)\le 7\pi.
$$

\end{proof}

\begin{thm}
Let $A\in {\cal A}_0$ be a unital separable simple $\mathcal{Z}$-stable C*-algebra with
unique tracial state.
Then $M(A \otimes \K)/(A \otimes \K)$ has real rank zero.
\label{thm:MainRR0}
\end{thm}

\begin{proof} Let $B =_{df} A \otimes K$, and
let $\pi : M(B) \rightarrow M(B)/B$ be the natural quotient map.
Let $J$ be the smallest ideal in $M(B)$ which properly
contains $B$.
Consider the six-term exact sequence:
\begin{equation} \label{FirstSix} \end{equation}
$$
\begin{array}{ccccc}
K_0(J) &\to & K_0(M(B)) & \to & K_0(M(B)/J)\\
\uparrow  &&&& \downarrow\\
K_1(M(B)/J)  &\leftarrow & K_1(M(B)) & \leftarrow & K_1(J)
\end{array}
$$
From Theorem \ref{K1}, $K_1(J)=0.$ The above six-term exact sequence implies that
the map $K_0(M(B)) \rightarrow K_0(M(B)/J)$ is surjective. But it is also known that $K_0(M(B))=0$
(\cite{BlackadarBook} Proposition 12.2.1;
see \cite{Mingo})
Hence,
$K_0(M(B)/J) = 0$. Since $M(B)/J\cong \pi(M(B))/\pi(J),$ $K_0(\pi(M(B))/\pi(J)) = 0.$
Therefore the map
$K_0(\pi(M(B))) \rightarrow K_0(\pi(M(B))/\pi(J))$ is surjective.

Since both $\pi(J)$ and $\pi(M(B))/\pi(J)$ are simple purely infinite
(this is well-known;
an explicit reference can be found in
\cite{KucerovskyNgPerera};  it also follows immediately from
\cite{LinFull} Theorem 3.5 and the definition of $J$ in
\cite{LinContScale} 2.2, Remark 2.9 and Lemma 2.1; see also \cite{ZhStruct} Theorem 2.2 and its
proof),
both $\pi(J)$ and $\pi(M(B))/\pi(J)$ have real rank zero.
It follows from  \cite{BrownPedersenRR0} Theorem 3.14 and Proposition 3.15 that
$\pi(M(B))$ has real rank zero.
\end{proof}

\begin{thm}
$M(\mathcal{Z} \otimes \K)/(\mathcal{Z} \otimes \K)$
has real rank zero.
\end{thm}

\begin{rem}
We note that Theorem \ref{thm:MainRR0} includes many other C*-algebras  (see \cite{LN-Range})--
including crossed products coming from uniquely ergodic minimal homeomorphisms
on a compact metric space with finite topological dimension (e.g., see
\cite{WinterZ},   \cite{ConnesDyn}.)
\end{rem}

Finally, we end the paper with some K-theory computations
that follow immediately from our work.

\begin{cor}
Let $\pi : M(\mathcal{Z} \otimes \mathcal{K}) \rightarrow
M(\mathcal{Z} \otimes \mathcal{K})/(\mathcal{Z} \otimes \mathcal{K}) =
Q(\mathcal{Z})$ be the natural quotient map.
Let $J$ be the unique smallest ideal of
$M(\mathcal{Z} \otimes \mathcal{K})$ which properly
contains $\mathcal{Z} \otimes \mathcal{K}$.

Then we have the following:
$K_0(J) = \mathbb{R}$, $K_1(J) = 0$,
$K_0(\pi(J)) = \mathbb{R}/\mathbb{Z} \cong \T$,
$K_1(\pi(J)) = 0$,
$K_0(Q(\mathcal{Z})/\pi(J)) = 0$
and $K_1(Q(\mathcal{Z})/\pi(J)) = \mathbb{R}$.
\end{cor}

\begin{proof}
Let $B =_{df} \mathcal{Z} \otimes \mathcal{K}$.

That $K_0(J) = \mathbb{R}$ follows immediately from
Proposition \ref{prop:TracialDetermination}
and Proposition \ref{L43}.

  That $K_1(J) = 0$ follows immediately from Theorem \ref{K1}.

  That $K_0(Q(\mathcal{Z})/\pi(J)) = 0$ was computed in the
proof of Theorem \ref{thm:MainRR0}.

   Hence, the six-term exact sequence (\ref{FirstSix})
from Theorem \ref{thm:MainRR0} thus becomes
the following six-term exact sequence:

$$
\begin{array}{ccccc}
\mathbb{R} &\to & 0 & \to & 0\\
\uparrow  &&&& \downarrow\\
K_1(M(B)/J)  &\leftarrow & 0 & \leftarrow & 0
\end{array}
$$

Hence, $K_1(Q(\mathcal{Z})/\pi(J)) = K_1(M(B)/J) = \mathbb{R}$.

Next, the exact sequence
$$0 \rightarrow B \rightarrow J \rightarrow J/B \rightarrow 0$$
results in the six term exact sequence

$$
\begin{array}{ccccc}
K_0(B) & \rightarrow & K_0(J) & \rightarrow & K_0(J/B) \\
\uparrow & & & & \downarrow \\
K_1(J/B) & \leftarrow & K_1(J) & \leftarrow & K_1(B)
\end{array}
$$
\noindent which, by the above results, becomes
$$
\begin{array}{ccccc}
\Z & \rightarrow & \mathbb{R} & \rightarrow & K_0(J/B) \\
\uparrow & & & & \downarrow \\
K_1(J/B) & \leftarrow & 0 & \leftarrow & 0
\end{array}
$$

Since the map $\Z \rightarrow \mathbb{R}$ is the natural inclusion,
$K_1(\pi(J)) = 0$ and $K_0(\pi(J)) = \mathbb{R}/\Z \cong \T$.
\end{proof}

\end{document}